\title{\bf{Classification of simple 0-dimensional isolated complete intersection singularities }}
\author{
       \bf{Thuy Huong Pham,  Gerhard Pfister, Gert-Martin Greuel}\\
 }
\date{}

\documentclass[12pt]{article}
\usepackage{amsmath, amsthm, amssymb, mathrsfs, amscd, amsthm, wasysym, amscd,amsfonts,graphicx}
\usepackage{indentfirst,url,xypic, xcolor}
\usepackage{lipsum} 
\usepackage[all]{xy}
\usepackage{url}
\usepackage[colorlinks,plainpages]{hyperref}
\usepackage{verbatim}
\usepackage{fancyvrb}
\usepackage{listings}

\usepackage[colorlinks,plainpages]{hyperref}
\hypersetup{
	colorlinks=true,
	linkcolor=blue,
	filecolor=magenta,      
	urlcolor=cyan,
	citecolor=blue
}

\DeclareMathOperator{\rank}{rank}
\DeclareMathOperator{\corank}{corank}

\DeclareMathOperator{\mng}{mng}

\DeclareMathOperator{\characteristic}{char}
\DeclareMathOperator{\order}{ord}

\newtheorem{Definition}{ Definition}[section]
\newtheorem{Theorem}[Definition]{Theorem}
\newtheorem{Remark}[Definition]{Remark}
\newtheorem{Proposition}[Definition]{Proposition }

\newtheorem{Lemma}[Definition]{Lemma}

\newcommand{\N}{\mathbb{N}}

\newcommand{\Z}{\mathbb{Z}}

\begin{document}
\maketitle
\section*{Abstract}
The aim of this article is the classification of  simple 0-dimensional isolated complete intersection singularities in positive characteristic. As usual, a singularity is called simple or 0-modal if there are only finitely many isomorphism classes of singularities into which the given singularity can deform. The notion of simpleness was introduced by V. I. Arnold and  the classification of low modality singularities has become a fundamental task in singularity theory. Simple complex analytic isolated complete intersection singularities (ICIS) were classified by M. Giusti. However, the classification in positive characteristic requires different methods and is much more involved. The final result is nevertheless similar to the classification in characteristic 0 with some additional normal forms in low characteristic. 

The theoretical results in this paper mainly concern families of ICIS that are formal in the fiber and algebraic in the base (formal deformation theory is not sufficient). In particular, we  give a definition of modality in this situation and prove its semicontinuity.

\section{A criterion for non-simpleness} \label{sec:1}
In this section we establish the notations, prove some preliminary results and show a general criterion for non-simplicity. This criterion (Theorem \ref{th.notsimple}) allows us to significantly reduce the number of cases to be analyzed. The classification then consists of two main tasks: a) to further reduce the list of potentially simple singularities and b) to prove that the singularities of the list are indeed simple (Theorem \ref{simple-table} and \ref{simple-char2}). In step a), the semicontinuity of the modality is an important tool.
In both steps we have to consider small characteristics separately (where characteristic 2 is treated in a separate section) and for the computations we make substantial use of the system {\sc singular} \cite{DGPS15}.\\

Let $K$ be an algebraically closed field of arbitrary characteristic $p\geq0$ and denote by 
$$R=K[[x_1,\ldots, x_n]]$$
 the formal power series ring in $n$ variables ${\bf x}=(x_1,\ldots, x_n)$ with maximal ideal $\mathfrak{m}=\langle x_1,\ldots, x_n\rangle$. For $f=\sum a_\alpha{\bf x}^\alpha\in R$, we denote by $f({\bf 0})$ the coefficient $a_0$ of $f$ and by 
 $supp(f) =\{\alpha \mid a_\alpha \ne 0\}$ the support of $f$. 
 For $I\subset R$ a proper ideal, denote by $\mng(I)=\dim_K I/\mathfrak{m}I$ the minimal number of generators of $I$.
Let $GL(m,R)$ the group of $m\times m$ invertible matrices over $R$ and $Aut(R)$ the group of $K$-algebra automorphisms of $R$. 

\begin{Definition}\rm
	\begin{enumerate}
\item A polynomial $f\in K[x_1,\ldots, x_n]$ is called {\it quasi-homogeneous} of type $(d; {\bf a})=(d;a_1,\ldots, a_n)\in \Z_+\times \Z_+^n$ if $f$ is a $K$-linear combination of monomials ${\bf x}^\alpha=x_1^{\alpha_1}\ldots x_n^{\alpha_n}$ satisfying
$\langle {\bf a},{\bf \alpha}\rangle:=\sum\limits_{i=1}^n a_i\alpha_i=d.$
Denote by $G_d^ {\bf a}$ the vector space of all quasi-homogeneous polynomials of type $(d; {\bf a})$. Then 
$$K[x_1,\ldots, x_n] = \mathop\oplus\limits_{d\ge 0} G_d^ {\bf a}.$$ 
We set $G_d^ {\bf a}=\{0\}$ if $d<0.$ 
\item An element $f=(f_1,\ldots, f_m)\in K[x_1,\ldots, x_n]^m$ is called {\it quasi-homogeneous} of type $( {\bf d}; {\bf a})=(d_1,\ldots, d_m; a_1,\ldots, a_n)\in \Z_+^m\times \Z_+^n$ if
$$f_i\in G_{d_i}^ {\bf a}, \hskip 10pt \forall i=1,\ldots, m.$$
$a_i$ is called the {\em weight} of $x_i$ and $d_i$ the {\em (weighted) degree} of $f_i$.
Set
$$G_\nu^{ {\bf d}, {\bf a}}:=\{(f_1,\ldots, f_m)\in K[x_1,\ldots, x_n]^m\mid f_i\in G^ {\bf a}_{d_i+\nu}, i=1,\ldots,m\},$$
which is a finite dimensional $K$-vector space and 
$K[x_1,\ldots, x_n]^m=\oplus_{\nu \in \Z} G^ {\bf d,a}_{\nu}$.
\end{enumerate}
\end{Definition}

\begin{Definition}\rm
Let $( {\bf d};  {\bf a})=(d_1,\ldots, d_m; a_1,\ldots, a_n)\in\Z_+^m\times \Z_+^n$.
\begin{enumerate}
\item  For $f\in R\smallsetminus\{0\}$ we define
$$v_{\bf a}(f):=\inf\{ \langle {\bf a},{\bf \alpha}\rangle\mid \alpha\in supp(f)\}$$
and call it the ${\bf a}${-\it order} of $f$, or just the {\em order} if {\bf a}=$(1,...,1)$, denoted by $\order_ {\bf a}(f)$ or $\order(f)$.
If $f=0$ we set $v_{\bf a}(f)=\infty$.
\item For $f=(f_1,\ldots, f_m)\in R^m$ we set
	$$v_{{\bf d}, {\bf a}}(f):=\inf\limits_{i=1,\ldots,m}\{v_ {\bf a}(f_i)-d_i\}.$$
If $f\ne 0$ then $v_{{\bf d}, {\bf a}}(f)=\inf\{\nu\mid f=\sum\limits_{\nu}f^{(\nu)},\hskip 5pt  0\ne f^{(\nu)}\in G^{{\bf d}, {\bf a}}_\nu\}.$
\end{enumerate}
\end{Definition}

\begin{Definition}\rm 
	(1) Let $f=(f_1,\ldots, f_m)\in R^m$. Denote by
	\[T^1(f):=R^m{\Big/}\left(\left\langle f_1,\ldots, f_m\right\rangle\cdot R^m+\left\langle\frac{\partial f}{\partial x_1},\ldots,\frac{\partial f}{\partial x_n} \right\rangle \cdot R\right),\]
$\frac{\partial f}{\partial x_i}=(\frac{\partial f_1}{\partial x_i},\ldots,\frac{\partial f_m}{\partial x_i}) \in R^m$,
	the {\it Tjurina module} of $f$ and
	$$\tau(f):=\dim_KT^1(f)$$ 
	the {\it Tjurina number} of $f$. 	

(2) Two elements $f, g \in R^m$ are called {\it left equivalent},  if they belong to the same orbit of the {\it left group}
\[G:=GL(m,R) \rtimes Aut(R),\]
\noindent acting on $R^m$ by
\[(U,  \phi, f)\mapsto U\cdot\phi(f),\]
with  $U\in GL(m,R),  \phi\in Aut(R)$ and 
$$\phi(f)=[f_{1}\left(\phi({\bf x})), \ldots, f_{m}(\phi({\bf x})\right)]^T,$$
 where $\phi({\bf x})=\left(\phi(x_1),\ldots,\phi(x_n) \right)$. 
\end{Definition}

\begin{Definition}\rm 
An ideal $I\subset R$ defines a  {\em complete intersection} if $I$ can be gnerated by $f_1,\ldots, f_m$ with $f_i \in \frak m$ for all $i$ such that $f_i$ is a non-zero divisor of $R/\langle f_1,...,f_{i-1}\rangle$ for $i=1,...,m$.
Then $\mng(I)=m$  and $\dim R/I = n-m \ge 0$.

Let	$f=(f_1,\ldots, f_m)$ and  $g=(g_1,\ldots, g_m)\in R^m$ be such that $I=\langle f_1,\ldots, f_m\rangle$ and $J=\langle g_1,\ldots, g_m\rangle$ define complete intersections  with $\mng(I)=\mng(J)=m$.  Then $f$ is called {\it contact equivalent} to $g$, denoted $f\sim g$, if  $g\in Gf$, 
 i.e. there is a matrix $U\in GL(m, R)$ and an automorphism $\phi\in Aut(R)$ such that
$$\left[
\begin{matrix}
g_1\\
\vdots\\
g_m
\end{matrix}
\right]	= U\phi\left( \left[
\begin{matrix}
f_1\\
\vdots\\
f_m
\end{matrix}
\right]\right).$$
For simplicity of notations, we also write $g=U\phi(f).$
\end{Definition}
Note that by  \cite[Proposition 3.3.4]{P16}, $f\sim g$ if and only if $I\mathop\sim\limits^c J$, i.e. there is an automorphism $\phi\in Aut(R)$ such that $\phi(I)=J$.
It follows easily that if $f\sim g$ then $\tau(f)=\tau(g)$.

\begin{Definition}\rm
Let	$f=(f_1,\ldots, f_m)\in R^m$ be such that $I=\langle f_1,\ldots, f_m\rangle$  defines a  complete intersection with $\mng(I)=m$.
We call $f$ (or $I$) an {\it isolated complete intersection singularity} (ICIS)  if there is a positive integer $k$ such that 
$$\mathfrak{m}^k\subset I+I_m(Jac(f)),$$
where $Jac(f)=\left[\frac{\partial f_i}{\partial x_j}\right]_{ij}$ is the $m\times n$ Jacobian matrix of $f$ and $I_m(Jac(f))$ denotes the ideal generated by all $m\times m$ minors of $Jac(f)$.
We set 
$$I_{m,n} :=\{ f=(f_1,\ldots, f_m)\in R^m \mid f  \text{ is an ICIS with } \mng(I)=m\} .$$
\end{Definition}	

Note that any ICIS $f$ is finitely determined  (see Proposition \ref{prop.semicont}),  in particular,  the $f_i$ can be chosen as polynomials (up to an automorphism of $R$).

\begin{Lemma}\label{homogeneous submodule}
	Let $k\in \mathbb{N}$ and $D=\langle h_1,\ldots, h_k\rangle\subset R^m,$ a  submodule, where $h_i\in G^{{\bf d}, {\bf a}}_{\nu_i}$, $i=1,\ldots, k$ are quasi-homogeneous. Then we have an isomorphism of $K$-vector spaces
	$$R^m/D\cong \mathop\prod\limits_{\nu\in\Z}\left(G^{{\bf d}, {\bf a}}_{\nu}/D\cap G^{{\bf d}, {\bf a}}_{\nu}\right).$$
\end{Lemma}
We omit the elementary proof.

\begin{Lemma}
	Let  $f\in  K[x_1,\ldots, x_n]^m\cap I_{m,n}$ be quasi-homogeneous of type $({\bf d}; {\bf a})=(d_1,\ldots, d_m; a_1,\ldots, a_n)$. Then, as $K$-vector spaces,
	$$T^1(f)\cong \mathop\oplus\limits_{\nu\in \Z}T^1_\nu(f),$$
	where $T^1_\nu(f)=G^{{\bf d}, {\bf a}}_{\nu}\Big/\left(\left\langle f_1,\ldots, f_m\right\rangle\cdot R^m+\left\langle\frac{\partial f}{\partial x_1},\ldots,\frac{\partial f}{\partial x_n} \right\rangle \right)\cap G^{{\bf d}, {\bf a}}_{\nu}$.
\end{Lemma}
\begin{proof}
Note that if  $f\in  K[x_1,\ldots, x_n]^m\cap I_{m,n}$ is quasi-homogeneous of type $({\bf d}; {\bf a})$ then for all $i=1,\ldots, n$, $j=1,\ldots,m$, $\frac{\partial f_j}{\partial x_i}$ is quasi-homogeneous of type $(d_j-a_i;a)$ (or zero). Therefore, $f_je_k\in G^{{\bf d}, {\bf a}}_{d_j-d_k}$, where $\{e_k\}_{k=1,\ldots,m}$ is the canonical basis of $R^m$, and $\frac{\partial f_j}{\partial x_i}\in G^{{\bf d}, {\bf a}}_{-a_i}$, i.e.
 $D:=\left\langle f_1,\ldots, f_m\right\rangle\cdot R^m+\left\langle\frac{\partial f}{\partial x_1},\ldots,\frac{\partial f}{\partial x_n} \right\rangle$ is a quasi-homogeneous submodule of $R^m.$ By Lemma \ref{homogeneous submodule},
$$R^m/D\cong \mathop\prod\limits_{\nu\in\Z}\left(G^{{\bf d}, {\bf a}}_{\nu}/D\cap G^{{\bf d}, {\bf a}}_{\nu}\right).$$
Since $\tau(f)<\infty,$ we have 
$R^m/D\cong \mathop\oplus\limits_{\nu\in\Z}\left(G^{{\bf d}, {\bf a}}_{\nu}/D\cap G^{{\bf d}, {\bf a}}_{\nu}\right).$
\end{proof}

The following proposition is quite useful for the classification of singularities. 

\begin{Proposition}\label{Merle}
	Let $f=(f_1,\ldots, f_m)\in K[x_1,\ldots, x_n]^m\cap I_{m,n}$ be quasi-homogeneous of type $({\bf d}; {\bf a})\in \Z_+^m\times \Z_+^n$. Let $g=(g_1,\ldots, g_m)\in R^m$ be such that
	$$v_{{\bf d}, {\bf a}}(g)>\beta=\sup\{0,\alpha\},$$
	where $\alpha=\sup\{ i\mid T^1_i(f)\ne 0\}$. 
Then $f + g \sim f$.
\end{Proposition} 
\begin{proof}  In the complex analytic setting a proof is given in \cite[Proposition 1]{Giu83}. 
Using the finite determinacy of an ICIS (see Proposition \ref{prop.semicont}), the proof works as well for formal power series in any characteristic. We omit the details.
\end{proof}

\begin{Definition} \label{def.defsec} \rm
Let  $f=(f_1, \ldots, f_m)\in I_{m,n}$. 
The {\it tangent image} at $f$ to the orbit of $f$ under the action of $G$ on $R^m$  is defined as the module
$$\tilde T_f(Gf):=\langle f_1, \ldots, f_m\rangle R^m+\mathfrak{m}\left\langle \frac{\partial f}{\partial x_1},\ldots,\frac{\partial f}{\partial x_n}\right\rangle \subset \frak m R^m,$$
which  has finite $K$-codimension in $R^m$ for an ICIS.
It is easy to see that 
 $$T^{1,sec}(f):= \frak m R^m/\tilde T_f(Gf)$$
  has finite $K$-dimension iff this holds for
 $T^1(f) = R^m / \langle f_1,\ldots, f_m \rangle R^m+ \langle \frac{\partial f}{\partial x_1},\ldots,\frac{\partial f}{\partial x_n} \rangle$ (\cite[Proposition 4.2]{GP18}). We denote by 
 $$\tau^{sec}(f):=\dim_KT^{1,sec}(f).$$
\end{Definition}

\begin{Remark} \label{rem.defsec}
{\em
(1) The tangent image was introduced in \cite{GP18} (in a more general context) to replace the tangent space to the orbit $Gf$ in positive characteristic. In characteristic zero $\tilde T_f(Gf)$ coincides with the tangent space $T_f(Gf)$, see \cite[Lemma 2.8]{GP18}.
In general we have
$\tilde T_f(Gf)\subset T_f(Gf)$ and  the inclusion
can be strict in positive characteristic (see Example 2.9 in  \cite{GP18}).

It was shown in \cite[Theorem 1.4 ]{GP19} that an ICIS $f$ is finitely determined iff $\tau^{sec}(f)<\infty$ (equivalently $\tau(f)<\infty$).

(2) Let  $g_1,\ldots, g_d\in K[[{\bf{x}}]]^m$ represent a $K$-basis of $ T^1(f)$, resp. of $T^{1,sec}(f)$, then 
 $$F_{\bf t}({\bf x}):=F({\bf x}, t_1, \ldots, t_d)=f({\bf x})+\sum\limits_{i=1}^d t_i g_i({\bf x}) \in K[[{\bf t},{\bf x}]]$$
represents a formally semiuniversal deformation of $f$, resp. a formally semiuniversal deformation with section of $f$.  See  \cite{KS72} for the first case. The latter statement follows because every section can be trivialized (see \cite[Proposition II.2.2]{GLS07}) and then the proof given in  \cite{KS72} or   \cite[Theorem II.1.16]{GLS07} can be adapted. If the $g_i$ represent a system of generators, then $F_{\bf t}$ represents a formally versal deformation, resp. with section. 

(3) Note that the $g_i$ can be chosen as monomials and the $f_i$ as polynomials (since the ICIS $f$ is finitely determined). Thus the semiuniversal deformation, resp. with section, of $f$ has an algebraic representative $F_{\bf t}({\bf x}) \in K[{\bf t},{\bf x}]$ with $F_{\bf 0}({\bf x}) =f({\bf x})$. $F_{\bf t}({\bf x})$ is called
an unfolding  at $0\in T$ of $f$ over $T=K^d$.

The algebraic unfolding
$F_{\bf t}({\bf x})$  is {\em G-complete} or {\em contact complete}, meaning that  any unfolding
$H_{\bf s}({\bf x})= H({\bf x}, s_1,...,s_e) \in K[{\bf s},{\bf x}]$ of $f$ at ${\bf s_0} \in S =K^e$  with 
$H_{s_0}({\bf x}) =f({\bf x})$  is an  \'{e}tale pullback of $F$, i.e. there exist an \'etale neighbourhood $\varphi:W,w_0\to S,s_0$ and a morphism $\psi:W,w_0\to T,0$ such that $H({\bf x},\varphi(w))$ is contact-equivalent to $F({\bf x},\psi(w))$ for all $w\in W$.
This is important for the classification of singularities in the non-analytic case (the proof of \cite[Proposition 2.14]{GNg16} for hypersurfaces can be generalized to an ICIS). }
\end{Remark}

For the classification we need to consider $k$-jets of power series. Denote by 
$$R^{m}_k=R^m/\mathfrak{m}^{k+1}R^m, k\ge 0,$$ 
the {\em $k$-jet space} of $R^m$
and for $f\in R^m$,  let $j_k(f)$ denote the {\it $k$-jet of $f$}, i.e.,  the image of $f$ in $R^{m}_k$.  Consider the group of $k$-jets
$$G_k=\{(j_k(U), j_k(\phi))\mid U\in GL(m,R),\phi\in Aut(R)\},$$
where $(j_k(\phi))(x_i)=j_k(\phi(x_i))$, $i=1,\ldots, n.$ 
 Then the action of the left group $G$ on  $R^m$ induces the action on the $k$-jet spaces
\begin{align*}
G_k\times R^{m}_k &\to R^{m}_k\\
(j_k(U),j_k(\phi), j_k(f))&\mapsto j_k(U\phi(f)), 
\end{align*}
which is an algebraic action 
of the algebraic group $G_k$ on the affine space $R^{m}_k$. \\

We recall the finite determinacy and the semicontinuity of the Tjurina number of a complete intersection in arbitrary characteristic. We consider families
 $F_{\bf t}({\bf x})=F({\bf x}, t_1, \ldots, t_k) \in K[\bf t][[\bf x]]$, which are polynomial in the parameter ${\bf t}$, and denote for fixed  ${\bf t} \in K^k$ by $F_{\bf t}({\bf x})$ also the power series in $K[[{\bf x}]]$.

\begin{Proposition} \label{prop.semicont}
(1) Let $f \in R^m$ be an ICIS. Then $f$ is {\it finitely determined}, i.e.  there exists a $k$ such that each $g \in  R^m$ with $j_k(g) = j_k(f)$ is contact equivalent 
to $f$ ($f$ is then called $k$-determined). Moreover,  $f$ is $(2\tau(f) - \order(f) + 2)$-determined and every deformation of $f$ is $(2\tau(f)+1)$-determined.

(2) Let $F_{\bf t}({\bf x}):=F({\bf x}, t_1, \ldots, t_k) \in K[\bf t][[\bf x]]$ s.t. $F_{{\bf t}_0}$ is an ICIS in $I_{m,n}$ for some ${\bf t}_0\in K^k$. Then there exists a Zariski open set $U \subset K^k$ s.t. $F_{\bf t} \in I_{m,n}$ for all ${\bf t }\in U$. Moreover, for each $\tau\geq 0$ the sets
$$U_\tau = \{{\bf t}\in U\mid \tau(F_{\bf t}) \leq \tau \} \text{ and }
U_\tau^{sec} = \{{\bf t}\in U\mid \tau^{sec}(F_{\bf t}) \leq \tau \}\ $$
are open in $K^k$. In particular, if $\tau_{min} = \text{min} \{\tau(F_{\bf t})\mid {\bf t}\in K^d \}$, then $U_{\tau_{min}}$  is open and dense in $K^k$, and similar for $\tau_{min}^{sec}$.
\end{Proposition}
\begin{proof} (1) follows from \cite[Theorem 4.6]{GP18}. 
(2) follows for $d=1$ from \cite[Proposition 3.4]{GP18} and for an arbitrary Noetherian ring $A$ instead of $K[\bf t]$ (also for non-closed points) from \cite[Proposition 3.4 and Corollary 2.7]{GPf21}.
\end{proof}

In our classification we will consider families of $k$-jets of an ICIS. These  are deformations with trivial section which are versal for sufficiently large $k$. Hence, in the following definition of simple we consider deformations with section.

\begin{Definition}\label{def.simple}\rm 
Let $g_1,\ldots, g_d\in K[{\bf{x}}]^m$ be a set of $K$-generators of 
$T^{1,sec}(f).$
	An element $f\in I_{m,n}$ is called {\it simple} if there is a finite set $\{h_1,\ldots, h_l\}\subset I_{m,n}$ such that for 
	$$F_{\bf t}({\bf x}):=F({\bf x}, t_1, \ldots, t_d)=f({\bf x})+\sum\limits_{i=1}^d t_i g_i({\bf x}),$$
	there is a Zariski open neighborhood $U$ of $0$ in $K^d$ such that for each ${\bf t}=(t_1,\ldots, t_d)\in U$, there is a $j\in\{1, \ldots, l\}$ such  that $F_{\bf t}\sim h_j.$
\end{Definition}

\begin{Remark}\label{rem.mod} {\em 
(1) By Remark \ref{rem.defsec} (2) and (3) the unfolding $F_{\bf t}({\bf x})\in K[{\bf t},{\bf x}]$ defines a formally versal deformation with section of $f$ which is contact complete.
It follows that $f$ is simple iff for   an arbitrary unfolding
$H_{\bf s}({\bf x})= H({\bf x}, s_1,...,s_e) \in K[{\bf s},{\bf x}]$ of $f$ at $s_0 \in S =K^e$  there is a Zariski open neighborhood $U$ of $s_0$ in $S$ such that for each ${\bf s} \in U$, there is a $j\in\{1, \ldots, l\}$ such  that $H_{\bf s}\sim h_j.$

(2) Following \cite{GK90}, we call an element $f\in \mathfrak m R^m$  {\it 0-modular} if for each $k$ (equivalently, for sufficiently large $k$),  there is a Zariski neighborhood $U_k$ of $j_k(f)$ in $R^{m}_k$ such that $U_k$ meets only finitely many $G_k$-orbits of $k$-jets.  Here, an integer $k$ is {\it sufficiently large} for $f$ with respect to $G$ if there exists a neighbourhood $U$ of $j_k(f)$ in  $\mathfrak m R^{m}_k$ such that every $g\in  \mathfrak m R^m$ with $j_k(g) \in  U$ is $k$-determined with respect to $G$.  
If $f$ is an ICIS, then $k=(2\tau(f)+1)$ is sufficiently large by Proposition \ref{prop.semicont} (1). 

(3) More generally, we define the modality of an ICIS $f$ as follows. Let  $U\subset R^{m}_k$ an open neighbourhood of $j_k(f)$. We set
\begin{eqnarray*}
U(i)&:=&\{ g \in U\ |\ \dim_g (U\cap G_k\cdot g ) =i\}, i\geq 0,\\
G_k\text{-}\mathrm{par}(U)&:=&\max_{i\geq 0}\{\dim U(i)-i\},
\end{eqnarray*}
and call
$G_k\text{-}\mathrm{par}(j_k(f)):=\min \{G\text{-}\mathrm{par}(U)\ |\ U \text{ a neighbourhood of } j_k(f)\}$
the {\em number of $G_k$-parameters} of $j_k(f)$ (in $R^{m}_k$). It can be shown that 
if $k$ is sufficiently large then $G_k\text{-}\mathrm{par}(j_k(f))$ is independent of $k$ and denoted by $G\text{-}\mathrm{par}(f)$ or
$G\text{-}\mathrm{mod}(f)$ and called the {\em $G$-modality} of $f$ (this shown in  \cite[Definition 2.3]{GNg16} for hypersurfaces; the proof can be adapted for an ICIS). $f$ is 0-modular iff $G\text{-}\mathrm{mod}(f)=0$.
}
\end{Remark}

\begin{Proposition} \label{prop.modular}
(1) An ICIS is simple iff it is 0-modular.\\
(2) The $G$-modality of an ICIS is semicontinuous. In particular, any deformation of a simple ICIS is again simple.
\end {Proposition}
\begin{proof}We only sketch the proof. (1) We can modify the proof in \cite[Proposition 2.14. and Corollary 2.17]{GNg16} (for hypersurfaces) and  use the semi-continuity of $\tau(f)$ and $\tau^{sec}(f)$ (see Proposition \ref{prop.semicont}), to show that an ICIS is simple iff it is 0-modular.\\
(2) It follows from Definition \ref{def.simple} and the openness of versality that any deformation of a simple ICIS is again simple. 
More general,  the $G$-modality of an ICIS is semicontinuous. This can be proved as for hypersurfaces, see  \cite[Proposition 2.7] {GNg16}.
\end{proof}

The following result is the main general condition for a 0-dimensional ICIS being not simple.

\begin{Theorem}\label{th.notsimple}
$f=(f_1,...,f_n) \in I_{n,n}$ is not simple if one of the following cases occurs:\\
a)  $n\ge 2$ and $\order(f)\ge 3$.\\
b) $n\ge 3$ and $\order(f) = 2$.
\end{Theorem}

\begin{proof} 
 Set $l=\order(f)=\min\{\order(f_i)\}$. Then  the  action of $G_l$ on $R^n_l$ induces an algebraic action of the affine algebraic group $G':=GL(n,K)\times GL(n,K) $ on the affine variety $X:=\mathfrak{m}^lR^n/\mathfrak{m}^{l+1}R^n$.  

Assume by contradiction that $f$ is simple. Since $f$ is 0-modular, there is a Zariski neighborhood $U_l$ of $j_l(f)$ in $R^n_l$ which intersects only finitely many $G_l$-orbits. This implies that there is Zariski neighborhood $U$ of $j_l(f)$ in $X$ which intersects only finitely many $G'$-orbits, say $U=\mathop\cup\limits_{i=1}^s(U\cap O_i)$. 

We get that there exists an orbit $G'h, h\in X,$ such that
$$n {n-1+l \choose l}=\dim X =\dim(U\cap X)=\max\limits_i\{\dim (O_i\cap U)\}\le \dim G'h$$  
(otherwise $X$ cannot be covered by finitely many orbits).\\
It is easy to see that the elements $\{(a^l E_n, (1/a)E_n)\mid a\in K^*\}$, with $ E_n$ 
the $n\times n$ identity matrix,  are in the stabilizer of $h$ under $G'$, hence 
$$\dim G' h \leq \dim G' -1 = 2 n^2 -1.$$
a) If  $n\ge 2$ and $l =3$  then $\dim X= \frac{n^2(n+1)(n+2)}{6} \geq 2n^2 >$
$2 n^2 -1$. If $l >3$ the l.h.s. gets bigger while the r.h.s. stays constant.  Hence
$\dim X > \dim G'h$ and $f$ cannot be simple for $n\geq 2$ and $l\geq 3$.\\
b) If $n\ge 3$ and $l=2$  then  $\dim X=\frac{n^2(n+1)}{2}>2n^2-1,$ showing that $f$ is not simple.

 \end{proof}
 
For further analysis we need the following splitting lemma in any characteristic.

\begin{Definition}\rm
	Let $f\in R$. We denote by 
	$$H(f)=\left[\frac{\partial^2 f}{\partial x_i\partial x_j}(0)\right]_{i,j=1,\ldots,n}\in Mat(n,K)$$
	the {\it Hessian matrix} of $f$ and by
	$$\corank(f)=n-\rank(H(f))$$
	the {\it corank} of $f$.  $f$ is called {\it non-degenerate} if corank($f$) = 0, otherwise it is called {\it degenerate}.
\end{Definition}

\begin{Lemma} [Right splitting lemma in characteristic $p\ne2$]\label{splitting lemma}
	Let $p\ne  2$ and $f\in\mathfrak{m}^2$ be such that $\corank(f)=k\ge 0.$ Then
	$$f\mathop\sim\limits^r g(x_1,\ldots, x_k)+x_{k+1}^2+\ldots +x_{n}^2,$$
	where $g\in \mathfrak{m}^3$ is uniquely determined up to right equivalence. 
	Of course, $x_i^2+x_j^2 \mathop\sim\limits^r x_ix_j$.
\end{Lemma}
\begin{proof}
 \cite[Lemma 3.9]{GNg16} and \cite[Theorem3.3.46]{GLS25} . 
\end{proof}

\begin{Lemma} [Right splitting lemma in characteristic 2]\label{splitting lemma char 2}
	Let $p=  2$ and $f\in\mathfrak{m}^2\subset K[[x_1,\ldots, x_n]]$, $n\ge 2$. Then there is an $l$, $0\le 2l\le n$, such that
	$$f\mathop\sim\limits^r x_1x_2+x_3x_4+\ldots+x_{2l-1}x_{2l}+ g (x_{2l+1},\ldots, x_n)$$
	where $g\in \langle x_{2l+1}, \ldots,x_{n} \rangle^3\subset K[[x_{2l+1}, \ldots,x_{n}]],$ or $g=x_{2l+1}^2+h $ with $h\in \langle x_{2l+1}, \ldots,x_{n} \rangle^3\subset K[[x_{2l+1}, \ldots,x_{n}]]$ if $2l<n$. $g$ is uniquely determined up to right equivalence.
\end{Lemma}
\begin{proof}
\cite[Proposition 3]{GK90}, \cite[Lemma 3.12]{GNg16}, and with full proof \cite[Theorem 3.3.51]{GLS25}.
\end{proof}


\section{Simple isolated complete intersection singularities 
in characteristic $p\ne2$}  \label{sec:2}
Let $f=(f_1,...,f_n)\in I_{n,n}$ be an ICIS. If $n=1$, then $f$ defines a hypersurface singularity $\mathop\sim\limits^c x^{k+1}$, which is simple of type $A_k$. Hence we assume $n\ge 2$. If $n\ge 3$ then $\order(f) = 1$ 
for $f$  simple (by Theorem \ref{th.notsimple}) and then $f\sim g\in I_{n-1,n-1}$ by the implicit function theorem. Thus we may assume $n=2$.
By Theorem \ref{th.notsimple} an ICIS $(f_1,...,f_n)\in I_{n,n}$ can be simple only if $n=2$ and $\order(f)=2$, which we assume from now on.
Let $R=K[[x,y]]$, $\mathfrak{m}=\langle x, y\rangle$ its maximal ideal, and  $p=\characteristic(K)$. 

\begin{Proposition}\label{non-degenerate}
	 Let $p\ge 0$. Let $f=(f_1, f_2)\in I_{2,2}$, with  $\order(f)=2$ and  $j_2(f_i)$ a non-degenerate quadratic form for some $i\in \{1, 2\}$. Then there are $m,n\ge 2$ such that $$f\sim (xy, x^n+y^m)=: F^{n,m}.$$
\end{Proposition}

\begin{proof}
By Lemma \ref{splitting lemma} and \ref{splitting lemma char 2}  we may assume that $f_1=xy$.
We can use $xy$ to kill in $f_2$ all terms divisible by $xy$. We may therefore assume that $f_2=u_1x^n+u_2y^m$
with $u_1\in K[[x]]$ and $u_2\in K[[y]]$ units ($u_1=0$ or $u_2=0$ is not possible, since $(f_1,f_2)$ is a complete intersection). We multiply $f_2$ with $u_1^{-1}u_2^{-1}$ and kill with $xy$ the terms of $u_1^{-1}y^m$ and 
$u_2^{-1}x^n$ and obtain $\alpha x^n+\beta y^m$ with $\alpha=u_2^{-1}(0)$ and $\beta=u_1^{-1}(0)\in K$ both different from $0$. Since $K$ is algebraically closed there exist the $n$-th root of $\alpha$ and the $m$-th root of $\beta$ and
we find an automorphism mapping the ideal $\langle xy, f_2\rangle$ to $\langle xy,x^n+y^m\rangle$.
\end{proof}

From now on in this section, we assume that $\characteristic(K)=p\ne 2$ unless otherwise stated.

\begin{Lemma}\label{lemma}
	Let $f,g\in \mathfrak{m}^2$. Assume that $j_2(f)$ is degenerate.
	Then there exist $s\ge 3$ and $\alpha\in \{0,1\}$  such that
	$$(f,g)\sim\left(x^2+\alpha y^s, \sum\limits_{i\ge t}a_iy^i+x\sum\limits_{j\ge q}b_jy^j\right),$$
	where  $t\ge 2$, $q\ge 1$ and $a_i, b_j \in K$.
\end{Lemma}

\begin{proof}
Using Lemma \ref{splitting lemma} we may assume that
$f=x^2+h(y)$ with $h=0$ or $h=y^s\cdot e$, $s\ge 3$, $e$ a unit. If $h\neq 0$, we can find (since $p\ne 2$) $e_1$ with $e_1^2=e^{-1}$. Using the automorphism $\varphi$ of $K[[x,y]]$ defined by $\varphi(x)=e_1x$, $\varphi(y)=y$ we may assume that $f=x^2+y^s$.
All together we have $f=x^2+\alpha y^s$ with $\alpha\in \{0,1\}$.
The 
Weierstra{\ss}  Division Theorem implies that $g=k\cdot (x^2+\alpha y^s)+\sum\limits_{i\ge t}a_iy^i+x\sum\limits_{j\ge q}b_jy^j$
for suitable $k\in K[[x,y]]$ and $a_i$, $b_j\in K$. This implies that
$$(x^2+\alpha y^s, g)\sim\left(x^2+\alpha y^s, \sum\limits_{i\ge t}a_iy^i+x\sum\limits_{j\ge q}b_jy^j\right).$$
\end{proof}

\begin{Lemma}\label{lemma_degenerate}
Let $f=(f_1,f_2)=(x^2+\alpha y^s, \sum\limits_{i\ge t}a_iy^i+x\sum\limits_{j\ge q}b_jy^j)$ be an ICIS such that 
$s\geq 3$, $t\geq 2$, $q\geq 1$ and $\alpha\in\{0,1\}$.
\begin{enumerate}
\item If $a_i=0$ for all $i$ and $b_q\neq 0$ then $\alpha =1$ and $f\sim (x^2+y^s,xy^q)$.
\item If $b_j=0$ for all $j$ and $a_t\neq 0$ then $f\sim (x^2+\alpha y^s,y^t)$. If $\alpha=0$ or $t\leq s$ then
 $f\sim (x^2,y^t)$.
 
\hspace{-1.6em} Assume now that $a_tb_q\neq 0$.
 \item If $t\leq q$ then $f\sim (x^2+\alpha y^s,y^t)$. If additionally $\alpha=0$ or $t\leq s$ then $f\sim (x^2,y^t)$.
 \item  If $t>q$ and $\alpha=0$ then $f\sim (x^2,y^t+xy^q)$.
 \item Let $t>q$ and $\alpha=1$. Then
 $f\sim (x^2+y^s,y^t+exy^q)$ for a suitable unit $e\in K[[y]]$.\\
  If $2t-2q-s\ne 0$\,\footnote{Note that for $f=(x^2+y^4,y^5+axy^3)$ we have $2t-2q-s=0$ and the parameter $a$ cannot be avoided by the proof of Proposition \ref{not simple2}.}  and $p\nmid 2t-2q-s$
 then $f\sim (x^2+y^s,y^t+xy^q)$.
\item If $t=q+1$
  and $(p = 0 \text{ or } p\nmid t)$ then $f\sim (x^2+\alpha y^s,y^t)$.\\
\end{enumerate}
\end{Lemma}
\begin{proof}
(1): $ f_2=x\sum\limits_{j\ge q}b_jy^j=uxy^q$, $u$ a unit. We get
$\langle f_1,f_2\rangle = \langle x^2+\alpha y^s,xy^q\rangle$. Since $f$ is an ICIS $\alpha$ must be $1$.\\

(2): $f_2=\sum\limits_{i\ge t}a_iy^i=uy^t, u$ a unit.
This implies $\langle f_1,f_2\rangle= \langle x^2+\alpha y^s,y^t\rangle$. If $\alpha=0$ or $t\leq s$ then obviously $\langle x^2+\alpha y^s,y^t\rangle= \langle x^2,y^t\rangle$.
\\

(3):If $t\leq q$ then $\sum\limits_{i\ge t}a_iy^i+x\sum\limits_{j\ge q}b_jy^j=y^t\cdot unit$, implying $ \langle f_1,f_2\rangle= \langle x^2+\alpha y^s,y^t\rangle$. If $\alpha=0$ or $t\leq s$ then  $\langle x^2+\alpha y^s,y^t\rangle= \langle x^2,y^t\rangle$.\\

(4): Multiplying $\sum\limits_{i\ge t}a_iy^i+x\sum\limits_{j\ge q}b_jy^j$ by the inverse of $\sum\limits_{i\ge t}a_iy^{i-t}$ we may assume that  $f=(x^2, y^t+e\cdot xy^q)$ with $e=\sum\limits_{j\ge q}b_jy^{j-q}$ a unit.
Consider the automorphism $\varphi$ of $K[[x,y]]$ defined by $\varphi(x)=e^{-1}x$ and $\varphi(y)=y$ then
$\varphi(f)=(e^{-2}x^2,y^t+xy^q)\sim (x^2,y^t+xy^q).$\\

(5): Multiplying $\sum\limits_{i\ge t}a_iy^i+x\sum\limits_{j\ge q}b_jy^j$ by the inverse of $\sum\limits_{i\ge t}a_iy^{i-t}$ we may assume that  $f=(x^2+ y^s, y^t+e\cdot xy^q)$ with $e=\sum\limits_{j\ge q}b_jy^{j-q}$ a unit. Let
\begin{center}
$F(z):=z^{s+2q-2t}\sum\limits_{j\ge q}b_jy^{j-q}z^{2(j-q)} - b_q .$
\end{center}
We have 
\begin{center}
$F(1)\in \langle y \rangle K[[y]]$ \\
$F'(1)=(s+2q-2t)\sum\limits_{j\ge q}b_jy^{j-q}+\sum\limits_{j\ge q+1}b_jy^{j-q}.$
\end{center}
Since $b_q\neq 0$ $F'(1)$ is a unit if $s+2q-2t\neq 0$ in $K$. Using Implicite Function Theorem\footnote{We apply the theorem (cf. Theorem 6.2.17 in \cite{GG} ) to the function $G(z):=F(z+1)\in K[[y,z]]$. We have $G(0)\in \langle y \rangle K[[y]]$ and $G'(0)\notin \langle y\rangle K[[y]]$ a unite. The theorem gives the existence of $\tilde z(y) \in \langle y\rangle K[[y]]$ such that $G(\tilde z(y))=0$. Let $z(y):=\tilde z(y)+1$ then $z(y)$ is a unit and $F(z(y))=0$.}
we obtain a unit $z(y)\in K[[y]]$ such that $F(z(y))=0$, i.e.
\begin{center}
$z(y)^{s+2q-2t}\sum\limits_{j\ge q}b_jy^{j-q}z(y)^{2(j-q)} =b_q .$
\end{center}
The map $\varphi$ defined by 
\begin{center}
$\varphi(x)=z(y)^s$ and $\varphi(y)=z(y)^2y$
\end{center}
gives the equivalence
\begin{center}
$(x^2+y^s,y^t+exy^q)\sim (x^2+y^s,y^t+ b_qxy^q).$
\end{center}
Consider the automorphism $\psi$ of $K[[x,y]]$ defined by $\psi(x)=ax$ and $\psi(y)=by$ with units
$a,b \in K[[x,y]]$. Then $\psi(f)=(a^2x^2+ b^sy^s,b^ty^t+b_qab^qxy^q)$.
We have to choose $a$ and $b$
such that 
\begin{itemize}
\item $a^2= b^s$,
\item $b^t=b_qab^q$.
\end{itemize}
We get the equation
$$b^{t-q-\frac{s}{2}}=b_q,$$
which  can be solved if $p\nmid 2t-2q-s \ne 0$.\\

(6): 
If $\alpha=0$, using (4) we may assume that $f=(x^2,y^t+xy^{t-1})$. Now $y^t+xy^{t-1}=(y+\frac{1}{t} x)^t+x^2h$ for a suitable
$h\in K[[x,y]]$. This implies that $\langle x^2,y^t+xy^{t-1}\rangle =\langle x^2,(y+\frac{1}{t} x)^t\rangle$. Using the automorphism $\varphi$ of $K[[x,y]]$ definde by $\varphi(x)=x$ and $\varphi(y)=y-\frac{1}{t} x$ we obtain the result.\\
If $\alpha=1$, using (5) we may assume that $f=(x^2+y^s,y^t+exy^{t-1})$, $n\geq 3$,  for a suitable unit $e\in K[[x,y]]$. Now $y^t+exy^{t-1}=(y+\frac{1}{t} ex)^t+x^2h$ for a suitable
$h\in \frak{m}^{t-2}$. Using the automorphism $\varphi$ of $K[[x,y]]$ definde by $\varphi(x)=x$ and $\varphi(y)=y-\frac{1}{t} ex$ we obtain $f\sim (x^2+(y-\frac{1}{t} ex)^s,y^t+x^2\bar h)$ with $\bar h\in\frak{m}^{t-2}$.
Now $\langle x^2+(y-\frac{1}{t} ex)^s,y^t+x^2\bar h \rangle = \langle x^2+(y-\frac{1}{t} ex)^s,y^t-(y-\frac{1}{t} ex)^s\bar h)\rangle = \langle x^2+y^s+\beta xy^{s-1}+\tilde k_1,y^t+\tilde h_1\rangle= \langle (x+\frac{1}{2}\beta y^{s-1})^2+y^s+\tilde k_1,y^t+\tilde h_1\rangle$ with $\beta\in K$, $\bar k_1,\tilde k_1\in\frak{m}^{s+1}$ and $ \tilde h_1\in\frak{m}^{t+1}$. Using the automorphism $\tilde \varphi$ of $K[[x,y]]$ definde by $\tilde\varphi(x)=x-\frac{1}{2}\beta y^{s-1}$ and $\tilde\varphi(y)=y$ we obtain $f\sim (x^2+y^s+k_1,y^t+h_1)$ with $k_1\in\frak{m}^{s+1}$ and $h_1\in \frak{m}^{t+1}$. 

We may assume that $k_1=x\sum\limits_{i\ge s}a_iy^i +\sum\limits_{i\ge s+1}b_iy^i$ and write
\begin{center}
$x2+y^s+k_1=(x+\frac{1}{2} \sum\limits_{i\ge s}a_iy^i)^2+y^s+\tilde k_1$, $\tilde k_1\in \langle y\rangle^{s+1}$.
\end{center}
Using the map $\psi$ defined by $\psi(x)=x-\frac{1}{2}\sum\limits_{i\ge s}a_iy^i$ and $\psi(y)=y$ we obtain
\begin{center}
$f\sim (x^2+y^s+\tilde k_1,y^t+\tilde h_1)$ with $\tilde h_1\in \frak{m}^{t+1}$.
\end{center}
Now $y^s+\tilde k_1=ey^s$ for a suitable unit $e\in K[[y]]$. The map $\varphi$ defined by $\varphi(x)=\sqrt(e)x$ and
$\varphi(y)=y$ gives the equivalence
\begin{center}
$(x^2+y^s+\tilde k_1,y^t+\tilde h_1)\sim (x^2+y^s,y^t+h)$ for a suitable $h\in\frak{m}^{t+1}$.
\end{center}
We may assume that $h=x\sum\limits_{i\ge t}a_iy^i+\sum\limits_{i\ge t+1}b_iy^i$. This implies that
$y^t+h=\tilde e y$ for a unit $\tilde e\in K[[x,y]]$, i.e. $(x^2+y^s,y^t+h)\sim (x^2+y^s,y^t)$.
\end{proof}

\begin{Proposition}\label{not simple2}

 Let     $\characteristic(K) = p$.
\begin{enumerate}
\item Assume that $p \ne 2$. 
Let $f= (x^2+ \alpha y^s,\sum\limits_{i\ge t}a_iy^i+x\sum\limits_{j\ge q}b_jy^j)$, $\alpha\in\{0,1\}$.\\
If $s\geq 4$, $t\geq 5$ and $q\geq 3$  then
$f$ is not simple.
\item  Assume that $p = 2$. Then $(x^2+axy,h), a\in K$, is not simple for $h\in K[[x,y]]$ with $\order(h)\geq 3$.
\end{enumerate} .
\end{Proposition}

\begin{proof}
The following {\sc Singular} computation shows that 
$(x^2+y^4,y^5+axy^3)$ is not
equivalent to $(x^2+y^4,y^5+bxy^3)$, except $a^2=b^2$. In particular  $(x^2+y^4,y^5+axy^3)$ is not simple for each $a\in K$.
We show below that this implies that $f$ is not simple if $\alpha=1$.
$f$ with $\alpha=0$ is not simple (by Proposition \ref{prop.modular} (2)), since it deforms to $f$ with $\alpha=1$ and $s\ge 4$.\\

{\sc Singular} computation to show that $(x^2+y^4,y^5+axy^3)$ is not simple: 
\begin{verbatim}
ring R=(0,a,b,l,m0,m1,m2,m3,m4,m5,m6,m7,n0,n1,n2,n3,n4,n5,n6,n7,o,s0,s1,
              t,u0,u1,v),(x,y),dp;
int i;
poly f=x2+y4;
poly g=y5+a*xy3;
poly h=x2+y4;
poly k=y5+b*xy3;
map phi=R,m0*y+m1*y2+m2*x+m3*xy+m4*y3+m5*y4+m6*x2+m7*xy2,
                   n0*y+n1*y2+n2*x+n3*xy+n4*y3+n5*y4+n6*x2+n7*xy2;
poly F=jet(phi(f)-(s0+s1*y)*h-t*k,5);
poly G=jet(phi(g)-(u0+u1*y)*h-v*k,5);
// R ist graded with deg(x)=2, deg(y)=1, we compute up to degree 5

matrix M1=coef(F,xy);
matrix M2=coef(G,xy);

ideal I;
for(i=1;i<=ncols(M1);i++){I[size(I)+1]=M1[2,i];}
for(i=1;i<=ncols(M2);i++){I[size(I)+1]=M2[2,i];}

ring S=0,(a,b,l,m0,m1,m2,m3,m4,m5,m6,m7,n0,n1,n2,n3,n4,n5,n6,n7,o,s0,s1,
               t,u0,u1,v,u,z),dp;
ideal I=imap(R,I);
I=I,(m0*n2-n0*m2)*o-1,(s0*v-u0*t)*z-1; //to make matrix and map invertible

matrix M=lift(I,ideal(2*(a2-b2)));
poly p; 
for(i=1;i<=nrows(M);i++){p=p+M[i,1]*I[i];}
p;
2*a^2-2*b^2

\end{verbatim}
Note that the ideal $I$ contains the conditions that  $(x^2+y^4,y^5+axy^3)$ is
equivalent to $(x^2+y^4,y^5+bxy^3)$ and the matrix $M$, computed by {\tt lift}, satisfies $p = M*I$ with $I$ considered as column vector. Printing out $I$ and $M$ we see that the coefficients are integers.  Hence, the equation $p = M*I$ is true as polynomials over $\mathbb Z$ (although the {\sc Singular} computation was done over $\mathbb Q$). This implies that in characteristic different from $2$ we obtain $a^2=b^2$.\\

Let $s\ge 4$, $t\ge 5$ and $q\geq3$.
Set $f_a := (x^2+y^4,y^5+axy^3), a\in K,$ and $F_\lambda := f+\lambda f_a$.  $F_\lambda$ is a deformation
of $f$ and we show that $F_\lambda$ is not simple for almost all   $\lambda$ (i.e. all except finitely many).
Then, by the semicontinuity of the modality (Proposition \ref{prop.modular}), $f$ is not simple. \\ 

$$ F_\lambda=(F_{\lambda_1}, F_{\lambda_2})  =(x^2(1+\lambda)+y^4(\lambda+\alpha y^{s-4}), y^5(\lambda+  \sum_{i\ge t}a_iy^{i-5}) +
xy^3(\lambda a +\sum_{j\ge q}b_jy^{j-3}).$$
For  general $\lambda$ this is equivalent to 
$$(x^2+y^4, y^5+xy^3(ab_0+\sum_{j\ge 1} b'_jy^j)),$$   
for some $b_0\ne 0$. In fact, replace $x$ by $\sqrt{1+\lambda}^{-1}x$, multiply $F_{\lambda,1}$ with  the inverse of $v=(\lambda+\alpha y^{s-4})$ and replace $x$ by $\sqrt{v}x$ and, finally, multipy $F_{\lambda,2}$ with the inverse of  $\lambda+  \sum\limits_{i\ge t-5}a_iy^i$. 
With  $a\ne 0$ the factor of $xy^3$ is a unit and we get 
$(x^2+y^4, y^5+xy^3b_0a+g)$, with $b_0\neq 0$ and  $g=x\sum\limits_{j\ge 4}b'_jy^j$. By the {\sc Singular} computation above $h_0:=(x^2+y^4, y^5+xy^3b_0a)$ is not simple. 

To see that $h_g:=(x^2+y^4, y^5+xy^3b_0a+g)$ is not simple, we cannot use the semicontinuity 
since $h_g$ is a deformation of $h_0$ with a priori smaller modality. But we can prove the non-simpleness in two ways:\\
a) It suffices that the 5-jet of $h_g$ is not simple. This follows from a similar (but longer\footnote{On a MacBook Air with M1 chip about 4,5 h and about 5 sec without the term $xy^4$.}) {\sc Singular} computation as above, which shows that
$(x^2+y^4,y^5+axy^3+cxy^4)$ is not equivalent to $(x^2+y^4,y^5+bxy^3+dxy^4)$, except $a^2=b^2$
for arbitrary $c,d$. \\
b) We notice that $h_0$ is quasi-homogeneous of type $(4,5;2,1)$  and it suffices to show that the 
weighted 5-jet of $h_g$ is not simple. We make an ''Ansatz" with a general (not quasi-homogeneous) coordinate transformation up to weighted degree 5, apply it to $(f,g)_{a,c}=(x^2+y^4,y^5+axy^3+cxy^4)$ and compare it with $(h,k)=:(f,g)_{b,d}$:
\begin{verbatim}
map phi=R,m0*y+m1*y2+m2*x+m3*xy+m4*y3+m5*y4+m6*x2+m7*xy2,
                   n0*y+n1*y2+n2*x+n3*xy+n4*y3+n5*y4+n6*x2+n7*xy2;
poly F=jet(phi(f)-(s0+s1*y)*h-t*k,5);
poly G=jet(phi(g)-(u0+u1*y)*h-v*k,5);
\end{verbatim}

Then $(f,g)_{a,c} \sim (f,g)_{b,d}$ implies $F=G=0$. Looking at $F$, we see that $\tt m0$ must be 0 and hence $\tt phi(cxy^4)$ has weighted degree bigger than 5. Therefore  $cxy^4$ can  be deleted and the first {\sc Singular} computation above suffices to show that $h_g$ is not simple.\\

In characteristic $2$ the following {\sc Singular} computation shows that 
$(x^2+axy, cxy^2+dy^3) \sim (x^2+bxy, cxy^2+dy^3)$ implies 
$a^4b^2c^4d+a^2b^4c^4d+a^4c^2d^3+b^4c^2d^3=0$, which happens, for
fixed $a,c,d$,  only for finitely many $b$.\\
It follows that $\langle x^2+axy, j_3(h) \rangle = \langle x^2+axy, cxy^2+dy^3\rangle$, $c,d$ suitable, is not simple.   Hence  $(x^2+axy, h)$ is not simple.
\end{proof}

\begin{verbatim}
ring R=(2,a,b,c,d,m0,m1,n0,n1,o,s,t,u,v),(x,y),dp;
int i;
poly f=x2+a*xy;        //a the parameter of the family
poly g=c*xy2+d*y3;     //generic element of degree 3
poly h=x2+b*xy;
poly k=c*xy2+d*y3;
map phi=R,m0*y+m1*x,n0*y+n1*x;
poly F=phi(f)-s*h-t*k;
poly G=phi(g)-u*h-v*k;

matrix M1=coef(F,xy);
matrix M2=coef(G,xy);

ideal I;
for(i=1;i<=ncols(M1);i++){I[size(I)+1]=M1[2,i];}
for(i=1;i<=ncols(M2);i++){I[size(I)+1]=M2[2,i];}

ring S=2,(a,b,c,d,m0,m1,n0,n1,o,s,t,u,v,u,z),dp;
ideal I=imap(R,I);
I=I,(m0*n1-n0*m1)*o-1,(s*v-u*t)*z-1;      

ideal F=eliminate(I,m0*m1*n0*n1*o*s*t*u*v*o*z);
F;
> F[1]=a^4*b^2*c^4*d+a^2*b^4*c^4*d+a^4*c^2*d^3+b^4*c^2*d^3
\end{verbatim}
Simple complex analytic isolated complete intersection singularities (ICIS) were classified by M. Giusti \cite{Giu83}. The following list of normal forms appears also in Giusti's list of 0-dimensional ICIS, but we get some extra normal forms for special characteristics. The notations follow basically those of M. Giusti \cite{Giu83}, where the upper index is the classical form of Giusti.
 
 \begin{Proposition}\label{main}
 If $p\ne 2$, then the following ICIS are the only candidates for being simple:
 \begin{enumerate}
 \item $F^{s,m}: \text{   } (xy,x^s+y^m)$, $s,m\geq 2$,
 \item $G_5^0: \text{   } (x^2,y^3)$,
   \item $G_5^1: \text{   } (x^2,xy^2+y^3)$ in $\characteristic p=3$,
 \item $G_7: \text{   } (x^2,y^4)$,
  \item $H_{s+3}: \text{   } (x^2+y^s,xy^2)$, $s\geq 3$, 
  \item $I_{2t-1}^0: \text{   } (x^2+y^3,y^t)$, $t\geq 4$,
 \item $I_{2t-1}^1: \text{   } (x^2+y^3,y^t+xy^{t-1})$, $t\geq 4$ if $p\mid t$,
 \item $I_{2q+2}^0: \text{   } (x^2+y^3,xy^q)$, $q\geq 3$,
\item $I_{2q+2}^1: \text{   } (x^2+y^3,xy^q+y^{q+2})$, $q\geq 3$ if $p\mid 2q+3$. 
 \end{enumerate}
 \end{Proposition}

 \begin{proof}
 Let $f=(f_1,f_2)$ be an ICIS.
 
(I) If $j_2(f_i)$ is non-degenerate for some $i$ we obtain $f\sim F^{s,m}$  for suitable $s,m\geq 2$, using Proposition \ref{non-degenerate}.\\
 
(II) If $j_2(f_i)$ is degenerate for some $i$ we obtain
 using Lemma \ref{lemma} that 
 \begin{center}
 $f\sim (x^2+\alpha y^s, \sum\limits_{i\ge t}a_iy^i+x\sum\limits_{j\ge q}b_jy^j), \ a_i,b_j\in K,$\\
with $s\ge 3$, $t\ge2$, $q \ge1$ and $\alpha\in\{0,1\}$.
\end {center}
Using Proposition \ref{not simple2} and that ord$(f) =2$, we obtain that if
 $f$ is simple, then
 \begin{center}
 $ s=3$ \ or \ $2\le t\leq 4$ \ or  \ $1\le q\leq 2$.\\
  \end{center}
 Now we use these inequalities and Lemma \ref{lemma_degenerate}.
 \begin{enumerate}
\item  Let $a_i=0$ for all $i$ and $b_q\neq 0$ (this we may assume since $f$ is a complete intersection). Then $f\sim (x^2+y^s,xy^q)$ ($\alpha=1$ since $f$ is a complete intersection). For $f$ to be simple we  have $s=3$ and then we get  $I^0_{2q+2}$ or we have $q=2$ and then we get $H_{s+3}$. 
\item Let $b_j=0$ for all $j$ and $a_t\neq 0$.  Then $f\sim (x^2+\alpha y^s,y^t)$.\\
$-$ If $\alpha=0$ or $t\leq s$ then $f\sim (x^2,y^t)$. Then  $t$ must be $2$, $3$ or $4$.
For $t=2$ we get $(x^2,y^2) \sim (xy,x^2+y^2) =  F^{2,2}$. For $t=3, 4$ we obtain $G^0_5, G_7$.\\
$-$ If $\alpha=1$ and $t>s$ then $s=3$, and we obtain $ I^0_{2t-1}$.
\item  Assume now that $a_tb_q\neq 0$.\\
  Lemma \ref{lemma_degenerate} implies:
 \begin{enumerate}
\item If $t\leq q$ then $f\sim (x^2+\alpha y^s,y^t)$.\\
$-$ If $\alpha=0$ or $t\leq s$ then $f\sim (x^2,y^t)$. We have $2\le t\leq 4$ and  we obtain  $F^{2,2}$ or $G^0_5$ or $G_7$.\\
$-$ If $\alpha=1$ and $t>s$  then $s=3$, and we obtain $I^0_{2t-1}$.
\item  If $t>q$ then,
 by  Lemma \ref{lemma_degenerate} (4) and (5) and
using the automorphism $\varphi $ of $K[[x,y]]$ defined by $\varphi(x) = e^{-1}x$ and $\varphi(y) = y,$ we obtain
\begin{alignat*}{5}
f&\sim  (x^2,y^t+xy^q), \text{ if } \alpha =0,\\
f&\sim  (x^2+y^s,y^t+exy^q)\sim  (x^2+e^2y^s,y^t+xy^q), \text{ if }   \alpha =1,
\end{alignat*}
where $e\in K[[y]]$ a unit with $e=1$ if  $2t-2q-s\ne 0$ and ($p=0$ or $p\nmid 2t-2q-s$).

\begin{enumerate}
\item Let $t=q+1$. \\
(i.1)  If $p\nmid t$, then $f\sim (x^2+\alpha y^s,y^t)$, by Lemma \ref{lemma_degenerate} (6).\\
\text{   } $-$  If $\alpha=0$ or $t\leq s$ then $f\sim (x^2,y^t)$,  $2\le t\leq 4$, and  we obtain  
$F^{2,2}$ or $G^0_5$ or $G_7$.\\
\text{   } $-$ If $\alpha=1$ and $t>s$  then $s= 3$, and we obtain $ I^0_{2t-1}$.\\
(i.2)  If $p\mid t$, then
 $f\sim (x^2+\alpha y^s,y^t+exy^{t-1})$  with $e\in K[[y]]$ a unit.  \\
\text{   } $-$ If $\alpha=0$ then $p=3$ and $t=3$ and 
by Lemma \ref{lemma_degenerate} (4) we get $G^1_{5}$. \\
\text{   } $-$ If $\alpha=1$ then either $s=3$ or ($t=3$ and $p=3$). \\  
(i.2.1) For the first case ($s=3$) we obtain
from Lemma \ref{lemma_degenerate} (5) 
 \begin{center}
  $f\sim   (x^2+ e^2y^3,y^t+xy^{t-1}),$
 \end{center} 
 (note that $p\mid t$ implies $p\nmid 2t-2q-s$ since $s=3$).\\
If $p=t=3$ we get $G^1_{5}$.
If $t\ge 4$ we get   $I^1_{2t-1}$.\\
(i.2.2)  To see the second case ($s>3$, $t=3$ and $p=3$) we start with
$f=(f_1,f_2)\sim  (x^2+e^2y^s,y^3+xy^2).$  Substracting $e^2y^{s-3}f_2$ from $f_1 $
we get
 \begin{center}
  $\langle x^2+e^2y^s,y^3+xy^2)\rangle = \langle x^2-e^2xy^{s-1},y^3+xy^2\rangle$ \\
  $ =  \langle (x-\frac{1}{2}e^2 y^{s-1})^2-\frac{1}{4} e^4y^{2s-2},y^3+xy^2\rangle$. 
\end{center}
Using the automorphism  
$\varphi(x)=x+\frac{1}{2}e^2 y^{s-1}$ and $\varphi(y)=y$ 
we obtain 
\begin{center}
$ (x^2-\frac{1}{4}e^4 y^{2s-2},y^3+xy^2+\frac{1}{2} e^2y^{s+1}) = 
  (x^2-\frac{1}{4}e^4 y^{2s-2},uy^3+xy^2) $
   \end{center}
  with $u=1+\frac{1}{2} e^2y^{s-2}$ a unit. \\
  Setting $\varphi(x)=ux$ and $\varphi(y)=y$ we obtain 
\begin{center}
$(u^2x^2-\frac{1}{4}e^4 y^{2s-2},uy^3+uxy^2)$
 $\sim (x^2-\frac{1}{4} u^{-2}e^4y^{2s-2},y^3+xy^{2}),$ 
 \end{center}
 which raises the power $s$ of $y$ to $2s-2$.
Iterating this process
 we obtain  $f\sim (x^2,y^3+xy^2)$
which is $G^1_5$.\\

\item  Now assume that $t>q+1$.\\ 
(ii.1) If  Let $\alpha=0$. Then Proposition \ref{not simple2} implies $q\leq 2$.\\
\text{   } $-$ If $q=1$, we obtain $(x^2,y^t+xy)$. Similar as below with $q=2$   this is equivalent to $(x^2+y^{2t-2},xy)$, which is $F^{2,2t-2}, t\ge 2$.\\
\text{   } $-$  If $q=2$ we obtain $(x^2,y^t+xy^2)$. But 
\begin{center}
 $(x^2,y^t+xy^2)=(x^2,y^2(y^{t-2}+x))$
 $\sim ((x-y^{t-2})^2,xy^2)$\\
 $=(x^2-2xy^{t-2}+y^{2t-4},xy^2)\sim (x^2+y^{2t-4},xy^2)$, since $t \ge 4$. 
 \end{center}
We get $H_{2t-1}$.\\
 
(ii.2) Let $\alpha=1$. Using 3.(b) above we can start with $(x^2+y^s,y^t+exy^q)$, with $e=e(y) = e_0 + e_ky^k + $ higher terms, where $e_0$ and $e_k$ are non-zero.\\
Then Proposition \ref{not simple2} implies 
 $s=3$ \ or \ $2\le t\leq 4$ \ or  \ $1\le q\leq 2$.\\
(ii.2.1)  Let $s\ge 3$ and $q\le 2$\\
\text{   } $-$ Let $q=1.$ Set $m=:\min\{s,2t-2\}$ if $s\neq 2t-2$ or if $s = 2t-2$ and $e^{-2}(0) \ne -1$. If $s = 2t-2$ and $e^{-2}(0) = -1$ set $m=:s+k$.

We get \\
 $(x^2+y^s,y^t+exy)\sim (x^2+y^s,y(e^{-1}y^{t-1}+x))$\\
 $ \sim ((x-e^{-1}y^{t-1})^2 +y^s,xy) \sim (x^2+e^{-2}y^{2t-2} +y^s,xy)$\\
 $ \sim (x^2+y^mu, xy)\sim (x^2u^{-1}+y^m,xy)$, $u(y)$ a unit.\\
 We get $F^{2,m}$ since $p\ne 2$. If $s=3$ we get $F^{2,3}$.\\
\text{   } $-$  Let $q=2$. Consider $(x^2+y^s,y^t+exy^2)$, 
$e(y)= e_0 + e_ky^k + $ higher terms,  a unit. 
Set $m=:\min\{s,2t-4\}$ if $s\neq 2t-4$ or if $s = 2t-4$ and $e^{-2}(0) \ne -1$. If $s = 2t-4$ and $e^{-2}(0) = -1$ set $m=:s+k$.\\
We get  $(x^2+y^s,y^t+exy^2)\sim (x^2+y^m,xy^2)$, i.e. $H_{m+3}$. 
\\ This follows from 
\begin{center}
$(x^2+y^s,y^t+exy^2)=(x^2+y^s,y^2(e^{-1}y^{t-2}+x))\sim ((x-e^{-1}y^{t-2})^2+y^s,xy^2)$
  $=(x^2-2e^{-1}xy^{t-2}+e^{-2}y^{2t-4}+y^s,xy^2)\sim (x^2+e^{-2}y^{2t-4}+y^s,xy^2)\sim (x^2+uy^m,xy^2)\sim
(u^{-1}x^2+y^m,xy^2)$, $u$ a unit, 
 $ \sim ((vx)^2+y^m,xy^2)$ $(v^2=u^{-1}) \sim (x^2+y^m,v^{-1}xy^2)$
 $\sim (x^2+y^m,xy^2)$.
  \end{center}
(ii.2.2) Let $s=3$ and $t\geq q+3$. We have
  $$y^t+exy^q=y^{t-3}(x^2+y^3)+(e-xy^{t-q-3})xy^q,$$
with $(e-xy^{t-q-3})$  a unit. Hence $(x^2+y^3,exy^q+y^t)\sim (x^2+y^3,xy^q)$.\\
(ii.2.3)  Let $s=3$ and $t=q+2$. \\
We consider $(f,g)=(x^2+y^3,xy^q+ey^{q+2})$ , using Lemma \ref{lemma_degenerate} (5), we
obtain $(f,g)\sim (x^2+y^3,xy^q+y^{q+2})$ since $2t-2q-s=1$.
 \end{enumerate}
 \end{enumerate}
 \end{enumerate}
\end{proof}

\begin{Theorem}\label{simple-table} Let $p=$ char$(K) \ne 2$. 
	An ICIS $f \in K[[x,y]]^2$ is simple iff it is contact equivalent to one of the normal forms in Table 2.\\
	$$
\begin{tabular}{|l |l l| }
	\hline
	     &&  \\[-2ex]
	Type & Normal form of $f$   & \\
	      & &  \\[-2ex]
	\hline
	        && \\[-2ex]
{$ F^{m,n}$,  \small{$m,n\ge 2$}}&{$ \hspace*{10mm}  (xy, x^m+y^n) $} &\\
       & & \\[-2ex]
	\hline
        && \\[-2ex]        {$ G_5$}&{$ G_5^0  \hspace*{5mm}  (x^2, y^3)$ }   and additionally  &\\
        && \\[-2ex]
		   & $ G_5^1   \hspace*{5mm}    (x^2, xy^2+y^3)$  \small{if $p=3$} &\\
        && \\[-2ex]
	\hline
	        & \\[-2ex]
	$ G_7$&$ \hspace*{10mm}  (x^2, y^4) $ &\\
       & & \\[-2ex]
	\hline      
	 & \\[-2ex]
	$ H_{n+3}$,  \small{$n\ge 3$}&$\hspace*{10mm} (x^2+y^n, xy^2) $& \\
        && \\[-2ex]	
	\hline       
	 & \\[-2ex]	
$ I_{2t-1}$, \small{ $t\ge 4$}&$ I_{2t-1}^0 \ (x^2+y^3, y^t) $ and additionally &\\
        && \\[-2ex]
       			& $ I_{2t-1}^1 \ (x^2+y^3, y^t+xy^{t-1}) $ \small{if $p\mid t$} &\\
        && \\[-2ex]
       		\hline
        && \\[-2ex]
	$ I_{2q+2}$, \small{$q\ge 3$} &$ I_{2q+2}^0 \ (x^2+y^3, xy^q) $ and additionally  &  \\
        && \\[-2ex]	
 &$  I_{2q+2}^1 \ (x^2+y^3, xy^q+y^{q+2})$  \small{if $p\mid 2q+3$} &\\ 
 [3pt] 
       	\hline
	\end{tabular} 
	$$
\centerline{\it Table 2}
\end{Theorem}

\begin{proof} By Proposition \ref{main} we have only to show that the normal forms $f$ in Table 2 are simple.
By Definition \ref{def.simple} we have to consider the semiuniversal deformation (with section) of $f=(f_1,f_2)$,
	$$F_{\bf t}({\bf x}):=F({\bf x}, t_1, \ldots, t_d)=f({\bf x})+\sum\limits_{i=1}^d t_i g_i({\bf x}),$$
where $g_1,...,g_d$ is a $K$-basis of 
$$T^{1,sec}(f)= \frak m K[[x,y]]^2/\langle f_1, f_2\rangle K[[x,y]]^2+\mathfrak{m}\langle {\partial f}/{\partial x},{\partial f}/{\partial y}\rangle.$$
 $f$ is simple if the set $\{F_{\bf t}({\bf x}), {\bf t} \in K^d\}$ decomposes into only finitely many contact classes. Moreover, by the argument at the beginning of Section \ref{sec:2} we need to consider only those $g_i$ with $\order(g_i) \ge2$
 (if $\order(g_i)=1$ for some $i$ we get an $A_k$ hypersurface singularity with bounded $k$).\\

To determine the  bases $B$ of $T^{1,sec}(f)$, we use two methods. For fixed characteristic $p$ the basis can be computed directly  with {\sc Singular}. In case that infinitely many $p$ are involved, we often use {\sc Singular} to compute $B$ for several fixed $p$ in order to get an idea of the final form. For a general proof we then do elementary (but sometimes tedious) transformations to get the final list of generators, for which we then prove linear independence. In the following we present only some of these transformations as an illustration, otherwise the paper would be too long.
\begin{enumerate}
\item $  F^{m,n}$,  $m,n\ge 2$ \\
$ \bullet~  {p\nmid m}~\textit{ or } { p\nmid n}:  (x,0)$, $ (y,0),$ 
$ (0,x),$  \ldots, $(0, x^{m-1}),$
$(0, y),$ \ldots, $(0, y^{n-1}).$\\
Let us assume that $p\nmid m$. 
We have $T^{1,sec}(F^{m,n})=\langle x, y\rangle R^2/D,$
where $D$ is the submodule generated by the following elements
$$(xy,0), (x^m+y^n,0), (0,xy), (0, x^m+y^n), (xy, mx^m), (x^2,nxy^{n-1}), (y^2, mx^{m-1}y).$$
We do some transformations to get 
$$(x^2,0)=(x^2,nxy^{n-1})-ny^{n-2}(0,xy)\in D,$$
$$(y^2,0)=(y^2,mx^{m-1}y)-mx^{m-2}(0,xy)\in D,$$
and $(xy,0)\in D$. Moreover, we  have
$$(0, x^m)=\frac{1}{m}(xy, mx^m)-\frac{1}{m}(xy, 0)\in D,$$
$$(0, y^n)=(0, x^m+y^n)-(0, x^m)\in D, $$
and $(0, xy)\in D.$ This implies $B$ is a set of generators of $T^{1,sec}(F^{m,n})$.

On the other hand, to see linear independence, assume that 
$$a_1(x,0)+a_2(y,0)+c_1(0,x)+\cdots +c_{m-1}(0, x^{m-1})+d_1(0,y)+\cdots+d_{n-1}(0, y^{n-1})\in D$$
where $a_1, a_2, c_1,\ldots, c_{m-1}, d_1,\ldots, d_{n-1}\in K$. Then there are $l_1, l_2,l_3,l_4,l_5,l_6, l_7\in R$ such that
\begin{align*}
&a_1(x,0)+a_2(y,0)+c_1(0,x)+\cdots +c_{m-1}(0, x^{m-1})+d_1(0,y)+\cdots+d_{n-1}(0, y^{n-1})\\ \notag
&=  l_1(xy,0)+l_2(x^m+y^n,0)+l_3 (0,xy)+l_4 (0, x^m+y^n)+l_5 (xy, mx^m)+\\ \notag
& +l_6(x^2,nxy^{n-1}) + l_7(y^2, mx^{m-1}y).\notag
\end{align*}
This implies
$$a_1x+a_2y=(l_1+l_5+l_8)xy+l_2(x^m+y^n)+l_6x^2+l_7y^2$$
and
\begin{align*}
& c_1x+\cdots+c_{m-1}x^{m-1}+d_1y+\cdots +d_{n-1}y^{n-1}=l_3xy+l_4(x^m+y^n)+ml_5x^m+\\
+ & nl_6xy^{n-1}+ ml_7x^{m-1}y.
\end{align*}
Since $m,n\ge 2$ the first equation implies $a_1=a_2=0$ and from the second equation we get that the remaining coefficients are zero. 
 
$ \bullet~ p\mid m, p\mid n  \textit{ and } m\ge n: $
 $ (x,0), (y,0) ,(0,x), \ldots, (0, x^{m}), $  $ (0, y), \ldots, (0, y^{n-1}). $\\
The proof is similar to the previous case.
%

$  F^{m,n}$ is of quasi-homogeneous type {$ (m+n, mn; n, m) $}.
\item 	{$ G_5^0 $}\\
$ \bullet~ p=3: (x,0), (y,0), (y^2, 0), $ 
	{$(0,x), (0, y), (0,xy),$} $ (0, y^2), (0, xy^2).$ \\
As mentioned above, the basis can be computed by using {\sc Singular}, since the characteristic $p=3$ is fixed.

	 $\bullet~ p\ge 5: (x,0), (y,0), (y^2, 0), $ 
	{$(0,x), (0, y), (0,xy),  (0, y^2)$.} 
	
We have
$T^{1, sec}(G^0_5)=\langle x,y\rangle R^2/D,$ where
$$D=\left\langle(x^2,0), (y^3,0),  (xy,0), (0, x^2), (0,xy^2), (0,y^3)\right\rangle.$$
Clearly $B$ is a set of generators of $T^{1, sec}(G^0_5)$.
The proof that $B$ is linear independent is easy.

\medskip	

        $ G_5^1 $ \\
$  (x,0), (y,0), (y^2, 0), $ 
	{$ (0,x), (0, y), (0,xy)$}, $(0, y^2)$.
	
	\smallskip

$ G_5$ is 	of quasi-homogeneous type {$ (2,3; 1, 1) $}.

\item 	$ G_7$\\
$ (x,0), (y,0), (y^2, 0), $
		{$(y^3,0), (0, x), (0,y),$}
	$(0, xy), (0, y^2), (0, xy^2), (0, y^3) $.

We have 
$T^{1, sec}(G_7)=\langle x,y\rangle R^2/D,$ where
$$D=\left\langle(x^2,0), (y^4,0),  (xy,0), (0, x^2), (0,xy^3), (0,y^4)\right\rangle.$$
Clearly $B$ is a set of generators of $T^{1, sec}(G_7)$. The proof of linear independence is again easy.

	
\smallskip

$ G_7$ is 	of quasi-homogeneous type {$ (2,4; 1, 1) $}.
	
\item 	$ H_{n+3} $, {$n\ge 3$} \\
 $(x,0), (y,0), \ldots, (y^{n-1}, 0), $\
	$(0,x), (0,y), (0,xy),  $ 
		$ (0,y^2), (0, y^3)  $.\\
We omit the proof, since it is similar to the above (but a bit more lengthy).		

$ H_{n+3} $ is 	of quasi-homogeneous type {$ (2n,n+4; n, 2) $}.
		
\item a) $ I_{2t-1}^0$,  $t\ge 4$.
We have $T^{1,sec}(I_{2t-1}^0)=\langle x, y\rangle R^2/D,$
where $D$ is the submodule generated by the following elements
$$(x^2+y^3,0), (y^t,0), (0, x^2+y^3), (0, y^t), (x^2,0), (xy,0), (3xy^2, txy^{t-1}).$$
$ \bullet ~ p\nmid t:  (x, 0), (y, 0), (y^2,0), $
	 $ (0,y), \ldots, (0,y^{t-1}),  (0,x)$, $(0, xy), \ldots, (0, xy^{t-2}) $ \\
We have $(x^2,0)\in D$, $(y^3,0)\in D$ and $(xy,0)\in D$. Moreover, $(0, y^t)\in D$, $(0, xy^{t-1})\in D$.

We omit the proof for the following cases.\\
$ \bullet ~ p\mid t: $ $(x,0), (y,0), (y^2,0),  (0,y), \ldots, (0, y^{t-1}), (0,x), (0, xy), \ldots, (0, xy^{t-1}).$

 b) $ I_{2t-1}^1$,  $t\ge 4$.\\ 
  $ \bullet ~ p\mid t$:   $(x,0), (y,0), (y^2,0),  (0,y), \ldots, (0, y^{t-1}), (0,x), 
(0, xy), \ldots, (0, xy^{t-2}).$ 

%
\smallskip

$ I^0_{2t-1}$ is 	of quasi-homogeneous type {$ (6,2t; 3, 2) $}.

\item a) $ I^0_{2q+2}$, $ q\ge 3$\\
$ \bullet ~ p\nmid 2q+3:   (x, 0), (y, 0), (y^2, 0)$,
 $(0,x), (0,y), \ldots, (0, y^{q+1})$,\\
\hspace*{22mm} $(0, xy), (0, xy^2), \ldots, (0, xy^{q-1})$\\
$ \bullet ~ p\mid 2q+3:$  $ (x, 0), (y, 0), (y^2, 0)$,
 $(0,x), (0,y), \ldots, (0, y^{q+2}),$\\
 \hspace*{22mm} $(0, xy), (0, xy^2), \ldots, (0, xy^{q-1})$\\ \smallskip
 
b) $ I^1_{2q+2}$, $ q\ge 3$, $p\mid 2q+3.$\\
 $(x, 0), (y, 0), (y^2, 0)$,
 $(0,x), (0,y), \ldots, (0, y^{q+1}), (0, xy), (0, xy^2), \ldots, (0, xy^{q-1})$\\ 

$ I^0_{2q+2}$ is 	of quasi-homogeneous type {$ (6,2q+3; 3, 2) $}.

\end {enumerate}

1. 	We prove that $F^{m,n}=(xy, x^m+y^n)$ is simple:\\
Since $j_2(f_1)$ is non-degenerate, this holds also for any deformation, which is then of type  $F^{p,q}$  for some $p,q$ by Proposition \ref{main}. Since $\tau^{sec}(f)$ depends only on $n$ and $m$ (by (1) above)  and since $\tau^{sec}$  is semicontinuous  by Proposition \ref{prop.semicont} there are only finitely many $F^{p,q}$  into which  $F^{m,n}$ deforms. Hence $f$ is simple.\\

 2. a) 	We prove that $G_5^0 =(x^2, y^3)$ is simple: \\
$\bullet ~ p= 3$: We have to consider 
	$$g=(x^2+ay^2, by^2+y^3+cxy+dxy^2),$$
	where $a, b, c, d\in K.$ 
	
	i) $c\ne 0$: By Proposition \ref{non-degenerate}, $g$ is of type $F$.
	
	ii)  $c= 0$: 
		
		\hskip 10pt 1) $a\ne 0$: By Proposition \ref{non-degenerate}, 	$g$ is of type $F$.
		
		\hskip 10pt 2)  $a= 0$: 
			
			\hskip 10pt 2.1) $b\ne 0$:	$g=(x^2, by^2+y^3+dxy^2).$	By Lemma \ref{lemma_degenerate}, , $g\sim(xy, x^2+y^2)=F^{2,2}.$	
					
			 \hskip 10pt 2.2) $b= 0$:	 $g=(x^2, y^3+dxy^2).$		
			 
			 	 \hskip 10pt 2.2.1) $d= 0$:		$g=G_5^0.$

			  \hskip 10pt 2.2.2) $d\ne0$:	By the proof of Proposition \ref{main}, $g\sim(x^2, y^3+xy^2)=G_5^1$.			  
			  		 			
\noindent$\bullet ~ p\ge 5$: Consider 
	$$g=(x^2+ay^2, by^2+y^3+cxy),$$
	where $a, b, c\in K.$ 
	
	i) $c\ne 0$: By Proposition \ref{non-degenerate}, there are $m,n\ge 2$ such that $g\sim F^{m,n}$.
	
		ii)  $c= 0$: 
		
		\hskip 10pt 1) $a\ne 0$: By Proposition \ref{non-degenerate}, there are $s,t\ge 2$ such that $g\sim F^{s,t}$.
		
			\hskip 10pt 2)  $a= 0$: 
			
			\hskip 10pt 2.1) $b\ne 0$:  Since $p\ne 2$, let $u\in K[[y]]$ be such that $u^2=y+b$. Consider the automorphism $\phi:R\to R$, $x\mapsto x$, $y\mapsto uy$. Then 
			$$\phi^{-1}(g)=\phi^{-1}(x^2, y^2u^2)=(x^2, y^2)\sim(xy, x^2+y^2)=F^{2,2}.$$
		 \indent\hskip 10pt 2.2) $b=0$: $g\sim G_5^0.$

\medskip		 
\medskip		 
2. b)  We prove $G_5^1=(x^2, xy^2+y^3)$ is simple. Consider
	$$g=(x^2+cy^2, ay^2+y^3+bxy+xy^2),$$
	where $a, b, c\in K$. 
	
	i) $c\ne 0$: By Proposition \ref{non-degenerate}, there are $m,n\ge 2$ such that $g\sim F^{m,n}$.
	
		ii) $c= 0$: $g=(x^2, ay^2+y^3+bxy+xy^2)$
		
		\hskip 10pt	1)  $a\ne 0$:
		
		\hskip10pt	1.1)  $b\ne 0$:  By Proposition \ref{non-degenerate}, there are $s,t\ge 2$ such that $g\sim F^{s,t}$.
		
		\hskip 10pt	1.2) $b= 0$: $g=(x^2, ay^2+y^3+xy^2)$
		
		\hskip 10pt	1.2.1) $a= 0$: $g=G_5^1.$ 
		
		\hskip 10pt	1.2.2) $a\ne 0$: By Lemma \ref{lemma_degenerate} , $g\sim F^{2,2}.$
		
	\hskip 10pt	2) $a= 0$: $g=(x^2, y^3+bxy+xy^2)$
	
	\hskip 10pt	2.1) $b= 0$: $g=G_5^1.$ 
	
	\hskip 10pt	2.2)  $b\ne 0$: By Proposition \ref{non-degenerate}, there are $k,l\ge 2$ such that $g\sim F^{k,l}$.\\
		
	3. We prove $G_7=(x^2, y^4)$ is simple. Consider
	$$g=(x^2+cy^2+dy^3,a_2y^2+a_3y^3+ y^4 +b_1xy+b_2xy^2), $$ 
	where $a_i, b_i, c, d\in K.$
	
	i) $b_1\ne 0$: By Proposition \ref{non-degenerate}, there are $m,n\ge 2$ such that $g\sim F^{m,n}$.
	
		ii)  $b_1=0$:
		
		\hskip 10pt	1)  $c\ne 0$: By Proposition \ref{non-degenerate}, there are $s,t\ge 2$ such that $g\sim F^{s,t}$.
		
		\hskip 10pt	2)  $c=0$: $g=(x^2+dy^3, a_2y^2+a_3y^3+ y^4+b_2xy^2).$
		
		\hskip 10pt	2.1)  $d\ne0$:  We may assume $d=1$.
		
			\hskip 10pt	2.1.1) $a_2\ne0$: By Lemma \ref{lemma_degenerate}, we get $g\sim F^{2,2}.$ 
			
				\hskip 10pt	2.1.2) $a_2=0$: $g=(x^2+y^3, a_3y^3+ y^4+b_2xy^2).$
				
				\hskip 10pt 2.1.2.1) $a_3\ne 0$:
				
					\hskip 10pt 2.1.2.1.1) $b_2\ne 0$: By the proof of Proposition \ref{main}, we get $g\sim (x^2, y^3)$ if $p\ge 5$,

						\indent\hskip 10pt  and $g\sim (x^2, y^3+xy^2)$ if $p=3.$ 
						
						\hskip 10pt 2.1.2.1.2) $b_2= 0$: By Lemma \ref{lemma_degenerate} , we have $g\sim (x^2, y^3)$.
		
	\hskip 10pt 2.1.2.2) $a_3= 0$:  $g=(x^2+y^3, y^4+b_2xy^2).$

	\hskip 10pt 2.1.2.2.1)	 $b_2\ne 0$: By Lemma \ref{lemma_degenerate}, we get
	$$g\sim H_6=(x^2+ y^3, xy^2).$$
	
			\hskip 10pt 2.1.2.2.2) $b_2=0$: $g=(x^2+y^3, y^4)=I^0_7.$	 
				
		\hskip 10pt	2.2) $d=0$: $g=(x^2, a_2y^2+a_3y^3+ y^4+b_2xy^2).$
		
			\hskip 10pt	2.2.1) $b_2\ne 0$:
			
			\hskip 10pt	2.2.1.1) $a_2\ne 0$: By Lemma \ref{lemma_degenerate}, $g\sim F^{2,2}$.
			
				\hskip 10pt	2.2.1.2) $a_2= 0$: $g=(x^2, a_3y^3+ y^4+b_2xy^2).$
				
				\hskip 10pt	2.2.1.2.1) $a_3\ne 0$: By the proof of Proposition \ref{main}, $g\sim(x^2,y^3)$ if $p\ge 5$ and 
				
				\indent\hskip 130pt$g\sim(x^2,y^3+xy^2)$ if $p=3$.
				
				\hskip 10pt	2.2.1.2.2) $a_3=  0$: By the proof of Proposition \ref{main}, there is $r\ge 4$ such that
				$$g=(x^2, y^4+b_2xy^2)\sim (x^2+y^r, xy^2)=H_{r+3}.$$
				
					\hskip 10pt	2.2.2) $b_2= 0$: $g=(x^2, a_2y^2+ a_3y^3+y^4).$
					
						\hskip 10pt
					2.2.2.1) $a_2\ne 0$: By Lemma \ref{lemma_degenerate}, $g\sim F^{2,2}$.
					
						\hskip 10pt
					2.2.2.2) $a_2= 0$:
					
						\hskip 10pt	2.2.2.2.1) $a_3\ne 0$:  By Lemma \ref{lemma_degenerate}, $g\sim(x^2, y^3).$
						
							\hskip 10pt	2.2.2.2.2) $a_3= 0$:  $g=G_7=(x^2, y^4).$\\

4. We prove $H_{n+3}=(x^2+y^n, xy^2)$, $n\ge 3$, is simple. Consider 
$$g=(x^2+c_2y^2+\ldots+c_{n-1}y^{n-1}+c_ny^n, a_2y^2+a_3y^3+bxy+xy^2),$$
where $c_n=1.$

i)  $b\ne 0$: By Proposition \ref{non-degenerate}, $g$ is of type $F$.

ii) $b= 0$:

\hskip 10pt 1) $c_2\ne 0$:  By Proposition \ref{non-degenerate}, $g$ is of type $F$.

\hskip 10pt 2) $c_2= 0$: $g=(x^2+c_3y^3+\ldots+c_{n-1}y^{n-1}+y^n, a_2y^2+a_3y^3+xy^2).$ 

\hskip 10pt 2.1) If $n=3$ then $g=(x^2+y^3, a_2y^2+a_3y^3+xy^2 )$. If $a_2\ne 0$, by Lemma \ref{lemma_degenerate}, $g\sim F^{2,2}.$ If $a_2=0$ and $a_3\ne 0$, by the proof of Proposition \ref{main}, $g\sim (x^2, y^3)$ if $p\ge 5$ and $g\sim (x^2, y^3+xy^2)$ if $p=3$.  If $a_2=a_3=0$  then $g=H_6.$

\hskip 10pt  2.2) If $n\ge 4$, set 
$$i=\min\{j\mid c_j\ne 0\}\ge 3.$$
Set $u=c_i+c_{i+1}y+\ldots+y^{n-i}$. Then $g\sim (u^{-1}x^2+y^i,  a_2y^2+a_3y^3+xy^2)$. Let $v\in K[[y]]$ be such that $v^2=u^{-1}$. Then $v$ is a unit. Consider the automorphism $\phi: R\to R$, $x\mapsto vx$, $y\mapsto y$. Then
 $$g\sim (x^2+y^i, \phi^{-1}(a_2y^2+a_3y^3+xy^2))=\left(x^2+y^i, a_2y^2+a_3y^3+\mathop\sum\limits_{i\ge 2}d_jxy^j\right),$$
 where $d_j\in K$ and $d_2\ne 0$.  

2.2.1) $a_2\ne 0$: By Lemma \ref{lemma_degenerate}, $g\sim F^{2,2}$.

2.2.2) $a_2= 0$:

2.2.2.1) $a_3\ne 0$: By the proof of Proposition \ref{main}, $g\sim (x^2, y^3)$ if $p\ge 5$ and 
 $g\sim (x^2, y^3+xy^2)$ if $p=3$.

2.2.2.2) $a_3= 0$: By Lemma \ref{lemma_degenerate}, $g\sim H_{i+3}$.\\

5. a)  We prove that $I^0_{2t-1}=(x^2+y^3, y^t)$, $t\ge 4$, is simple. \\
The proofs for the cases $p\nmid t$ and $p\mid t$ are put together as below.

Consider the unfolding
$$g=(x^2+y^3+cy^2, a_{2}y^2+\cdots+a_{t-1}y^{t-1}+a_{t}y^t+b_{1}xy+\cdots+b_{t-1}xy^{t-1}),$$
where $a_{t}=1$, and  $b_{t-1}=0$ if $p\nmid t$.

i)  $b_1\ne 0$: By Proposition \ref{non-degenerate}, $g$ is of type $F$.

 ii) $b_1= 0$:

\hskip 10pt 1) $c\ne 0$:  By Proposition \ref{non-degenerate}, $g$ is of type $F$.

\hskip 10pt 2) $c= 0$: $g=(x^2+y^3, a_2y^2+\ldots+a_{t-1}y^{t-1}+a_ty^t+b_2xy^2+\ldots+b_{t-1}xy^{t-1}).$

\hskip 10pt 2.1) $a_2\ne 0$:  By Lemma \ref{lemma_degenerate}, $g\sim(x^2, y^2) \sim F^{2,2}$.

\hskip 10pt 2.2) $a_2= 0$: $g=(x^2+y^3, a_3y^3+\ldots+a_{t-1}y^{t-1}+a_ty^t+b_2xy^2+\ldots+b_{t-1}xy^{t-1}).$

\hskip 10pt 2.2.1) $a_3\ne 0$: 

  \hskip 10pt 2.2.1.1)  $b_2\ne 0$: By the proof of Proposition \ref{main},  $g\sim (x^2, y^3)$ if $p\ge 5$, and $g\sim (x^2, y^3+xy^2)$ if $p= 3$.
  
  \hskip 10pt 2.2.1.2)  $b_2=0$: By Lemma \ref{lemma_degenerate}, $g\sim (x^2, y^3)$.
  
   \hskip 10pt 2.2.2)  $a_3= 0$: Set 
   $T=\min\{i\mid a_i\ne 0\}$. Then $4\le T\le t.$  
   
   \hskip 10pt 2.2.2.1) $b_2\ne 0$: By  the proof of Proposition \ref{main}, $g\sim H_{r+3}$ for some $r\ge 3$.
   
    \hskip 10pt 2.2.2.2)  $b_2= 0$: 
   
\noindent If  $b_j$ are not all 0 then we set  
  $$ Q=\min\{j\mid b_j\ne 0\}.$$
  Then $Q\ge 3$. We consider the following cases:
  
 \hskip 10pt 2.2.2.2.1) $T\le Q$: by Lemma \ref{lemma_degenerate}, $g\sim (x^2+y^3, y^T)= I^0_{2T-1}$.	
 
 \hskip 10pt 2.2.2.2.2) $T=Q+1$:	if $p\nmid T$, by Lemma \ref{lemma_degenerate}, $g\sim I^0_{2T-1}$.	If $p\mid T$, by the proof of Proposition \ref{main}, $g\sim I^1_{2T-1}$.

 \hskip 10pt 2.2.2.2.3) $T\ge Q+3$:  by the proof of Proposition \ref{main}, $g\sim I^0_{2Q+2}$. 
 
 \hskip 10pt 2.2.2.2.4) $T=Q+2$:  by the proof of Proposition \ref{main}, $g\sim I^0_{2Q+2}$ if $p\nmid 2Q+3$, and $g\sim I^1_{2Q+2}$ if $p\mid 2Q+3$.
 
\noindent If $b_j=0$ for all $j$ then by Lemma \ref{lemma_degenerate}, $g\sim I^0_{2T-1}$.

\medskip

5. b)  We prove that $I^1_{2t-1}=(x^2+y^3, y^t+xy^{t-1})$, $t\ge 4$, $p\mid t$, 
is simple.\\
Consider the unfolding
$$g=(x^2+y^3+cy^2, a_{2}y^2+\cdots+a_{t-1}y^{t-1}+a_{t}y^t+b_{1}xy+\cdots+b_{t-1}xy^{t-1}),$$
where $a_{t}=b_{t-1}=1$.

i)  $b_1\ne 0$: By Proposition \ref{non-degenerate}, $g$ is of type $F$.

 ii) $b_1= 0$:

\hskip 10pt 1) $c\ne 0$:  By Proposition \ref{non-degenerate}, $g$ is of type $F$.

\hskip 10pt 2) $c= 0$: $g=(x^2+y^3, a_2y^2+\ldots+a_{t-1}y^{t-1}+a_ty^t+b_2xy^2+\ldots+b_{t-1}xy^{t-1}).$

\hskip 10pt 2.1) $a_2\ne 0$:  By Lemma \ref{lemma_degenerate}, $g\sim(x^2, y^2) \sim F^{2,2}$.

\hskip 10pt 2.2) $a_2= 0$: $g=(x^2+y^3, a_3y^3+\ldots+a_{t-1}y^{t-1}+a_ty^t+b_2xy^2+\ldots+b_{t-1}xy^{t-1}).$

\hskip 10pt 2.2.1) $a_3\ne 0$: 

  \hskip 10pt 2.2.1.1)  $b_2\ne 0$: By the proof of Proposition \ref{main},  $g\sim (x^2, y^3)$ if $p\ge 5$, and $g\sim (x^2, y^3+xy^2)$ if $p= 3$.
  
  \hskip 10pt 2.2.1.2)  $b_2=0$: By Lemma \ref{lemma_degenerate}, $g\sim (x^2, y^3)$.
  
   \hskip 10pt 2.2.2)  $a_3= 0$:  Set 
   $$T=\min\{i\mid a_i\ne 0\}.$$
   Then $4\le T\le t.$  
   
   \hskip 10pt 2.2.2.1) $b_2\ne 0$: By the proof of Proposition \ref{main}, $g\sim (x^2+y^r, xy^2)= H_{r+3}$ for some $r\ge 3$.
   
    \hskip 10pt 2.2.2.2)  $b_2= 0$: Set  
  $$ Q=\min\{j\mid b_j\ne 0\}.$$
  Then $Q\ge 3$. We consider the following cases:
  
 \hskip 10pt 2.2.2.2.1) $T\le Q$: by Lemma \ref{lemma_degenerate}, $g\sim (x^2+y^3, y^T)= I^0_{2T-1}$.	
 
 \hskip 10pt 2.2.2.2.2) $T=Q+1$:	if $p\nmid T$, by Lemma \ref{lemma_degenerate}, $g\sim I^0_{2T-1}$.	If $p\mid T$, by the proof of Proposition \ref{main}, $g\sim I^1_{2T-1}$.

 \hskip 10pt 2.2.2.2.3) $T\ge Q+3$:  by the proof of Proposition \ref{main}, $g\sim I^0_{2Q+2}$. 
 
  \hskip 10pt 2.2.2.2.4) $T= Q+2$: by the proof of Proposition \ref{main}, $g\sim I^0_{2Q+2}$ if $p\nmid 2Q+3$, and $g\sim I^1_{2Q+2}$ if $p\mid 2Q+3$ .
\medskip

6. a) We prove $I_{2q+2}^0=(x^2+y^3, xy^q)$, $q\ge 3$, is simple. \\
The proofs for the cases $p\mid 2q+3$ and $p\nmid 2q+3$ are put together as below.\\
Consider the unfolding
$$g=(x^2+y^3+cy^2, a_2y^2+\ldots+a_{q+2}y^{q+2}+b_1xy+\ldots+b_{q-1}xy^{q-1}+b_qxy^q),$$
where $b_q=1$, and $a_{q+2}=0$ if $p\nmid 2q+3.$

i) $b_1\ne 0:$ By Proposition \ref{non-degenerate}, $g$ is of type $F$.

ii) $b_1=0:$ 

\hskip 10pt 1) $c\ne 0:$ By Proposition \ref{non-degenerate}, $g$ is of type $F$.

\hskip 10pt 2) $c=0:$

\hskip 10pt  2.1) $a_2\ne 0:$  by Lemma \ref{lemma_degenerate},  $g\sim F^{2,2}$.

\hskip 10pt  2.2) $a_2=0:$ we consider the following cases:

\hskip 10pt 2.2.1) $a_3\ne 0:$  

\hskip 10pt 2.2.1.1) $b_2\ne 0:$  
by the proof of  Proposition \ref{main},  $g\sim (x^2, y^3)$ if $p\ge 5$,

 \hskip 10pt and $g\sim (x^2, y^3+xy^2)$ if $p= 3$.

\hskip 10pt 2.2.1.2) $b_2=0:$  by Lemma \ref{lemma_degenerate},  $g\sim (x^2, y^3)$.

\hskip 10pt 2.2.2) $a_3= 0:$  

\hskip 10pt 2.2.2.1) $b_2\ne 0$:  by Lemma \ref{lemma_degenerate} and the proof of Proposition \ref{main},  $g\sim H_{r+3}=(x^2+y^r, xy^2)$ for some $r\ge 3$.

\hskip 10pt2.2.2.2) $b_2=0$: Set $Q=\min\{j\mid b_j\ne 0\}$. Then $3\le Q\le q.$

 If there is $i\ge 4$ such that $a_i\ne 0$ then we set 
 $$T=\min\{i\mid a_i\ne 0\}.$$
 Then $T\ge 4$.

\hskip 10pt 2.2.2.2.1) $T\le Q$: by Lemma \ref{lemma_degenerate}, $g\sim I^0_{2T-1}=(x^2+y^3, y^T)$. 

\hskip 10pt 2.2.2.2.2) $T=Q+1$: by the proof of Proposition \ref{main}, $g\sim I^0_{2T-1}$ if $p\nmid T$,  and $g\sim (x^2+y^3, x^T+xy^{T-1})=I^1_{2T-1}$ if $p\mid T.$

\hskip 10pt 2.2.2.2.3)  $T\ge Q+3$:  by the proof of Proposition \ref{main}, $g\sim I^0_{2Q+2}$. 

\hskip 10pt 2.2.2.2.4) $T=Q+2$:  by the proof of Proposition \ref{main}, $g\sim I^0_{2Q+2}=(x^2+y^3, xy^Q)$ if $p\nmid 2Q+3$, and $g\sim I^1_{2Q+2}=(x^2+y^3, xy^Q+y^{Q+2})$ if $p\mid 2Q+3$.
 
 If $a_i=0$ for all $i$ then by  Lemma \ref{lemma_degenerate},  we obtain $g\sim I^0_{2Q+2}$.  
 \medskip
 
6.  b) We prove $I_{2q+2}^1=(x^2+y^3, xy^q +y^{q+2})$, $q\ge 3$, $p \mid 2q+3$  is simple.\\
 Consider the unfolding
$$g=(x^2+y^3+cy^2, a_2y^2+\ldots+a_{q+2}y^{q+2}+b_1xy+\ldots+b_{q-1}xy^{q-1}+b_qxy^q),$$
where $a_{q+2}=b_q=1$.

i) $b_1\ne 0:$ By Proposition \ref{non-degenerate}, $g$ is of type $F$.

ii) $b_1=0:$ 

\hskip 10pt 1) $c\ne 0:$ By Proposition \ref{non-degenerate}, $g$ is of type $F$.

\hskip 10pt 2) $c=0:$

\hskip 10pt  2.1) $a_2\ne 0:$  by Lemma \ref{lemma_degenerate},  $g\sim (x^2, y^2)\sim F_{3}^{2,2}$. 

\hskip 10pt  2.2) $a_2=0:$ we consider the following cases:

\hskip 10pt 2.2.1) $a_3\ne 0:$  

\hskip 10pt 2.2.1.1) $b_2\ne 0:$  
by the proof of  Proposition \ref{main},  $g\sim G^0_5=(x^2, y^3)$ if $p\ge 5$,

 \hskip 10pt and $g\sim G^1_5= (x^2, y^3+xy^2)$ if $p= 3$.

\hskip 10pt 2.2.1.2) $b_2=0:$  by Lemma \ref{lemma_degenerate},  $g\sim (x^2, y^3)$.

\hskip 10pt 2.2.2) $a_3= 0:$  Set $$T=\min\{i\mid a_i\ne 0\}.$$
 Then $T\ge 4$.

\hskip 10pt 2.2.2.1) $b_2\ne 0$:  by the proof of Proposition \ref{main},  $g\sim H_{r+3}=(x^2+y^r, xy^2)$ for some $r\ge 3$.

\hskip 10pt2.2.2.2) $b_2=0$: Set $Q=\min\{j\mid b_j\ne 0\}$. Then $3\le Q\le q.$

\hskip 10pt 2.2.2.2.1) $T\le Q$: by Lemma \ref{lemma_degenerate}, $g\sim I^0_{2T-1}=(x^2+y^3, y^T)$. 

\hskip 10pt 2.2.2.2.2) $T=Q+1$: by the proof of Proposition \ref{main}, $g\sim I^0_{2T-1}$ if $p\nmid T$,  and $g\sim (x^2+y^3, x^T+xy^{T-1})=I^1_{2T-1}$ if $p\mid T.$

\hskip 10pt 2.2.2.2.3)  $T\ge Q+3$:  by the proof of Proposition \ref{main}, $g\sim I^0_{2Q+2}$. 

\hskip 10pt 2.2.2.2.4) $T=Q+2$:  by the proof of Proposition \ref{main}, $g\sim I^0_{2Q+2}=(x^2+y^3, xy^Q)$ if $p\nmid 2Q+3$, and $g\sim I^1_{2Q+2}=(x^2+y^3, xy^Q+y^{Q+2})$ if $p\mid 2Q+3$.
  \end{proof}


\section{Simple isolated complete intersection singularities of $I_{2,2}$ in characteristic 2}
In this section, we assume that the field $K$ has characteristic $p=2$.
\begin{Lemma}
	Let $f\in I_{2,2}$ be such that $\order(f)=2$. Then   either  there is a $g_2\in \mathfrak{m}^2$ such that
	$$f\sim (xy, g_2)$$
 or there are  $h_1\in \mathfrak{m}^3$  and $g_2\in \mathfrak{m}^2$ such that
$$f\sim (x^2+h_1, g_2).$$
If $g_2\in\frak{m}^3$ then $f$ is not simple.
\end{Lemma}
\begin{proof}
Let	$f=(f_1, f_2)$. We may assume that $\order(f_1)=2$. By Lemma \ref{splitting lemma char 2}, either there is $\phi\in Aut(R)$ such that $\phi(f_1)=xy$ or there is $\varphi\in Aut(R)$ such that $\varphi(f_1)=x^2+h_1$ for some $h_1\in\mathfrak{m}^3$. Therefore, either $f\sim (xy, \phi(f_2))$ or $f\sim (x^2+h_1, \varphi(f_2))$. 
Proposition \ref{not simple2} implies that $f$ is not simple if $g_2\in\frak{m}^3$ .
This proves the lemma.
\end{proof}

\begin{Proposition}
	Let $g=(xy, g_2)\in I_{2,2}$, where $g_2\in\mathfrak{m}^2$. Then there are $m,n\ge 2$ such that $g\sim (xy, x^n+y^m)=F^{n,m}$.
\end{Proposition}
\begin{proof}
	The proof is similar to a part of the proof of Proposition \ref{non-degenerate}, which does not depend on characteristic.
\end{proof}

\begin{Proposition}
	Let $g=(x^2+h_1, g_2)\in I_{2,2}$, where $h_1\in\mathfrak{m}^3$ and $j_2(g_2)=ax^2+bxy+cy^2\ne 0$, where $a,b, c\in K$.
	\begin{enumerate}
	 \item [a)] If $b=0$ and $c\ne 0$ then $g\sim (x^2, y^2)$.
	  \item [b)] If $b\ne 0$ and $c\ne 0$  then $g\sim (x^2, xy+ y^2)$.
	\end{enumerate}
\end{Proposition}
\begin{proof}

The proposed normal forms are of quasi-homogeneous type $({\bf d}; {\bf a})=(2,2;1,1)$. 

a) If $b=0$ and $c\ne 0$
then $g\sim (x^2+h_1, y^2+h_2)$,  $h_i\in \frak{m}^3$. Setting 
$$f=(x^2, y^2), \hskip 5pt h=(h_1, h_2),$$
 we have  $v_{{\bf d}, {\bf a}}(h)\ge 1$ and
	$T^1(f)=R^2/\langle x^2, y^2\rangle R^2$.
	It is easy to check $T^1_\nu(f)=0$ for all $\nu\ge 1$ (see below).  This implies
	$\sup\{i\mid T^1_i(H)\ne 0\}\le 0<v_{{\bf d}, {\bf a}}(h)$ and by
 Proposition \ref{Merle} we deduce $f\sim f+h\sim g$.\\

b) If $b\ne 0$ and $c\ne 0$ then
$g\sim (x^2+h_1, xy+y^2+h_2)$,  $h_i\in \frak{m}^3$. Setting
	$$f=(x^2, xy+y^2), \hskip 5pt h=(h_1, h_2).$$
we have  $v_{{\bf d}, {\bf a}}(h)\ge 1$. Moreover,
$T^1(f)=R^2/D,$
with $$D:=\langle(0,y), (0,x)\rangle+\langle x^2,  xy+y^2\rangle R^2.$$
  Then $(x^2,0)$, $(xy^2,0) $, $(y^3,0)$, $(0, xy^2)$, $(0, y^3)\in D.$ 
  We prove that for $\nu\ge 1$
  $$\left(\sum\limits_{i+j=2+\nu}c_{ij}x^iy^j, \sum\limits_{h+k=2+\nu}d_{hk}x^hy^k\right)\in D.$$
 Let $(i,j)\in \N^2$ be such that $i+j=2+\nu.$ If $i\ge 2$ then $(x^iy^j,0)\in D$. If $i=1$ then $j\ge 2$ so that $(xy^j,0)\in D$. If $i=0$ then $j\ge 3$ so that we get $(y^j,0)\in D$. Similarly, let $(h,k)\in \N^2$ be such that $h+k=2+\nu.$   If $h\ge 1$
   then $(0, x^hy^k)\in D$. If $h=0$ then $k\ge 3$ so that we get $(0, y^k)\in D.$
  This implies 
	$T^1_\nu(f)=0$ for all $\nu\ge 1$. By Proposition \ref{Merle}, $f\sim f+h \sim g.$ \end{proof}
	
\begin{Theorem} \label{simple-char2}
All ICIS of the  table 3 are simple and these are the only ones if $p=$ char$(K) = 2$.
$$
\begin{tabular}{|l |l l| }
\hline
	     &&  \\[-2ex]
Type & Equation of $f$ &   \\
	     &&  \\[-2ex]
\hline
	     &&  \\[-2ex]
$ F^{m,n}$ \small{$m,n\ge 2$}&$ (xy, x^m+y^n) $& \\
	     &&  \\[-2ex]
	     \hline
	     &&  \\[-2ex]
$F^{2,2;0}$ &$ (x^2, y^2) $&  \\
	     &&  \\[-2ex]
\hline
	     &&  \\[-2ex]
$F^{2,2;1}$ &$ (x^2, xy+y^2) $& \\
 [3pt] 
\hline
\end{tabular}
$$
	\centerline{Table 3}
\end{Theorem}

\begin{proof}To prove that the $f$ from table 3 are simple, we proceed as in the proof of Theorem \ref{simple-table}. We have the following bases of $T^{1,sec}(f)$:
\begin{enumerate}
\item $ F^{m,n}$, { $m,n\ge 2$}\\
$\bullet~ m~ \textit{or } n \text{ odd:} ~(x,0), (y,0),$ 
 $  (0,x), \ldots, (0, x^{m-1}), (0, y), \ldots, (0, y^{n-1}).$\\
$\bullet$ $m, n$ even, $m\ge n$:  $(x,0), (y,0),$
 $(0,x), \ldots, (0, x^{m}),  (0, y), \ldots, (0, y^{n-1}). $
\smallskip

$  F^{m,n}$ is of quasi-homogeneous type {$ (m+n, mn; n, m) $}.
\item $ F^{2,2;0}$ \\
$ (x,0), (y,0), (xy,0),  (0,x),  (0, y), (0,xy). $ 
\smallskip

$ F^{2,2;1}$\\
$ (x,0),  (y, 0), (y^2,0).$
\smallskip

$ F^{2,2}$ is of quasi-homogeneous type $ (2,2;1,1)$.
\end {enumerate}

1. The proof is similar to the proof of Proposition \ref{non-degenerate}, which works also for $p=2$.\\

2.	We prove that $ F^{2,2;0}=(x^2, y^2)$ is simple.

\hskip 10pt  Consider	$g=(x^2+axy, y^2+bxy)$ with $a, b\in K.$
	
\hskip 10pt i) $a= 0$: $g=(x^2, y^2+bxy)$

\hskip 20pt 1) $b=0$: $g=(x^2, y^2)$.

\hskip 20pt 2) $b\ne0$: Using the automorphism $\phi:$ $x\mapsto b^{-1}x$, $y\mapsto y$ we get

 \hskip 20pt  $\phi(g)=(b^{-2}x^2, xy+y^2) \sim (x^2, xy+y^2).$

\hskip 10pt ii) $a\ne0$:  Using the automorphism $\phi:$ $x\mapsto x$, $y\mapsto - a^{-1}x + a^{-1}y$, we get 

 \hskip 20pt  $\phi(g)=(xy, (a^{-2}-a^{-1}b)x^2+a^{-1}bxy+a^{-2}y^2)\sim (xy, (a^{-2}-a^{-1}b)x^2+a^{-2}y^2).$

\hskip 20pt 1) $a^{-1}\ne b$: $g\sim (xy, x^2+y^2)=F^{2,2}.$

\hskip 20pt 2) $a^{-1}= b$: $g\sim (xy, y^2)$ is not a complete intersection.\\

We prove that $F^{2,2;1}=(x^2, xy+y^2)$ is simple. 

\hskip 10pt  Consider $g=(x^2+ay^2, xy+y^2)$, where $a\in K$. Using $\phi: x\mapsto x-\sqrt{a}y$, $y\mapsto y$ 

\hskip 10pt we get $\phi(g)=(x^2, (1-\sqrt{a})y^2+xy)$.

\hskip 10pt i) $a= 1$: $g$ is not a complete intersection.

\hskip 10pt ii) $a\ne1$: Let $c=1-\sqrt{a}$. Using the automorphism
$\varphi: x\mapsto \sqrt{c}x,$ $y\mapsto c^{-\frac{1}{2}}y$

\hskip 10pt we get
$g\sim \varphi(\phi(g))=(cx^2, xy+y^2)\sim (x^2, xy+y^2).$
\end{proof}
{\bf{Funding:}} This research is funded by Vietnam Ministry of Education and Training (MOET) under grant number B2024-DQN-02.\\

{\bf{Acknowledgements:}} We would like to thank Shing Toung Yau and his students Hao Zuo and Hong Rui Ma for several questions that helped to improve the presentation.


\providecommand{\bysame}{\leavevmode\hbox to3em{\hrulefill}\thinspace}
\providecommand{\MR}{\relax\ifhmode\unskip\space\fi MR }
\providecommand{\MRhref}[2]{%
	\href{http://www.ams.org/mathscinet-getitem?mr=#1}{#2}
}
\providecommand{\href}[2]{#2}

\noindent 
Thuy Huong Pham\\
Department of Mathematics and Statistics, Quy Nhon University, 170 An Duong Vuong, Quy Nhon, Vietnam\\
Email address: phamthuyhuong@qnu.edu.vn

\medskip
\noindent Gerhard Pfister\\
Department of Mathematics, RPTU University of Kaiserslautern-Landau, Gottlieb-Daimler-Straße, 67663 Kaiserslautern, Germany\\
Email address: pfister@mathematik.uni-kl.de

\medskip
\noindent Gert-Martin Greuel\\
Department of Mathematics, University of Kaiserslautern, Germany\\
Email address: greuel@mathematik.uni-kl.de
\end{document}